\documentclass[12pt]{amsart}
\usepackage{amssymb}
\usepackage{bm}
\usepackage{graphicx}
\usepackage{psfrag}
\usepackage[usenames,dvipsnames]{xcolor}
\usepackage{enumerate}
\usepackage{multirow}
\usepackage{url}

\usepackage{subcaption}

\usepackage{comment}

\usepackage{algpseudocode}

\usepackage{mathtools}

\usepackage{xy}
\input xy
\xyoption{all}

\usepackage{tikz}

\newcommand{\excise}[1]{}

\newtheorem{thm}{Theorem}[section]
\newtheorem{lemma}[thm]{Lemma}

\newtheorem{prop}[thm]{Proposition}
\newtheorem{conj}[thm]{Conjecture}
\newtheorem{question}[thm]{Question}

\theoremstyle{definition}

\newtheorem{example}[thm]{Example}
\newtheorem{remark}[thm]{Remark}
\newtheorem{defn}[thm]{Definition}

\newtheorem{notation}[thm]{Notation}

\numberwithin{equation}{section}

\renewcommand\>{\rangle}
\newcommand\<{\langle}

\newcommand\QQ{\mathbb{Q}}

\newcommand\ZZ{\mathbb{Z}}

\newcommand\til{\mathord\sim}

\DeclareMathOperator\Betti{Betti} 

\DeclareMathOperator\Ap{Ap} 
\DeclareMathOperator\supp{supp} 

\usepackage{tikz-cd}

\begin{document}

\mbox{}
\title[On the cardinality of minimal presentations]{On the cardinality of minimal presentations \\ of numerical semigroups}

\author[Elmacioglu]{Ceyhun Elmacioglu}
\address{Mathematics Department\\Lafayette College\\Easton, PA 18042}
\email{elmacioglu.ceyhun@gmail.com}

\author[Hilmer]{Kieran Hilmer}
\address{Mathematics Department\\Purdue University\\West Lafayette, IN 47907}
\email{hilmerk@purdue.edu}

\author[O'Neill]{Christopher O'Neill}
\address{Mathematics Department\\San Diego State University\\San Diego, CA 92182}
\email{cdoneill@sdsu.edu}

\author[Okandan]{Melin Okandan}
\address{Mathematics Department\\Koç University\\ Istanbul, Turkey}
\email{melinokandan98@gmail.com}

\author[Park-Kaufmann]{Hannah Park-Kaufmann}
\address{Mathematics Department and Conservatory\\Bard College\\Annandale-on-Hudson, NY 12504}
\email{parkkaufmann@gmail.com}

\date{\today}

\begin{abstract}
In this paper, we consider the following question:\ ``given the multiplicity $m$ and embedding dimension $e$ of a numerical semigroup $S$, what can be said about the cardinality $\eta$ of a minimal presentation of $S$?''  We approach this question from a combinatorial (poset-theoretic) perspective, utilizing the recently-introduced notion of a Kunz nilsemigroup.  In addition to making significant headway on this question beyond what was previously known, in the form of both explicit constructions and general bounds, we provide a self-contained introduction to Kunz nilsemigroups that avoids the polyhedral geometry necessary for much of their source material.  
\end{abstract}

\maketitle



\section{Introduction}
\label{sec:intro}

A \emph{numerical semigroup} is a cofinite subset $S \subseteq \ZZ_{\ge 0}$ that is closed under addition and contains $0$.  We often specify a numerical semigroup using generators $n_1 < \cdots < n_k$,~i.e.,
\[
S = \<n_1, \ldots, n_k\> = \{z_1 n_1 + \cdots + z_k n_k : z_i \in \ZZ_{\ge 0}\}.
\]
It is known that each numerical semigroup $S$ has a unique minimal generating set, the elements of which are called \emph{atoms}; we write $\mathsf e(S) = k$ for its cardinality and $\mathsf m(S) = n_1$ for its smallest element, called the \emph{embedding dimension} and \emph{multiplicity} of $S$, respectively.  A \emph{factorization} of an element $n \in S$ is an expression 
$$n = z_1 n_1 + \cdots + z_k n_k$$
of $n$ as a sum of generators of $S$, which we often encode as a $k$-tuple $(z_1, \ldots, z_k)$.  

One of the primary ways of studying a numerical semigroup $S$ is via a minimal presentation $\rho \subset \ZZ_{\ge 0}^k \times \ZZ_{\ge 0}^k$, each element of which is a pair of factorizations that represents a minimal \emph{relation} or \emph{trade} between the generators of $S$ (we save the formal definition for Section~\ref{sec:overview}).  For example, if $S = \<6,9,20\>$, then 
$$\rho = \{((3,0,0), (0,2,0)), \, ((4,4,0), (0,0,3))\}$$
is a minimal presentation of $S$ consisting of 2 trades, the first between the factorizations $18 = 3 \cdot 6 = 2 \cdot 9$, and the second between the factorizations $60 = 4 \cdot 6 + 4 \cdot 7 = 3 \cdot 20$.  While a given numerical semigroup $S$ can have numerous minimal presentations, all have identical cardinality~\cite{fingenmon}; we denote this value by $\eta(S)$.  

\begin{question}\label{q:mainquestion}
Given the multiplicity $\mathsf m(S)$ and embedding dimension $\mathsf e(S)$ of a numerical semigroup $S$, what are the attainable minimal presentation cardinalities $\eta(S)$?  
\end{question}

Given a numerical semigroup $S$, some bounds are known for $\eta = \eta(S)$ in terms of $m = \mathsf m(S)$ and $e = \mathsf e(S)$.  It is known that $e - 1 \le \eta$, with equality if and only if $S$ is complete intersection~\cite{completeintersection}.  On the other hand, no upper bound for $\eta$ in terms of $e$ is possible in general.  Indeed, if $e = 2$, then $\eta = 1$, and if $e = 3$, then $\eta = 2$ when $S$ is complete intersection and $\eta = 3$ otherwise.  However, when $e = 4$ or larger, the value of $\eta$ can be arbitrarily large~\cite{primeidealsgenericzero}; see~\cite{monomialcurvecmtype,genericlatticecodim3} for families achieving these values, as well as~\cite{nsbettisurvey} for a survey of such results.  

More is known if one considers the value of $m$.  It is well known that $\eta \le \binom{m}{2}$, with equality if and only if $e = m$, in which case we say $S$ has \emph{max embedding dimension} (see \cite[Section~8.4]{numerical}).  
Some extensions of this are given in~\cite{highembdim}:\ if $e = m - 1$, then $\eta \in [\binom{e}{2} - 1, \binom{e}{2}]$, and if~$e = m - 2$, then $\eta \in [\binom{e}{2} - 2, \binom{e}{2}]$.  Additionally, if $3 \le e$, then $\eta = \binom{e}{2}$ is attained for~every $m \ge e$ by~\cite{rosalesApery}.  

In regards to Question~\ref{q:mainquestion}, the authors of~\cite{highembdim} note that their aforementioned results for $e \in [m-2, m]$ fail to extend to $e = m-3$, and remark this ``makes one think of alternative ways of study for numerical semigroups with not so high embedding dimension''.  
Recent work~\cite{kunzfaces3} has done just that, uncovering a new way to approach Question~\ref{q:mainquestion} that is poset-theoretic in nature.  The idea is to associate to each numerical semigroup $S$ a finite, partly cancellative nilsemigroup $N$, called the Kunz nilsemigroup of $S$, from whose divisibility poset the value $\eta(S)$ can be recovered.  

One of the primary difficulties in classifying $\eta(S)$ is that, except for a handful of specific families of numerical semigroups, minimal presentations can vary widely in structure.  Proving that a given set of trades is a minimal presentation usually involves a highly technical argument, and often requires a strong description of how one can navigate the factorizations of every element of $S$ using the given trades.  That is what makes Kunz nilsemigroups so advantageous for this task:\ the value of $\eta(S)$ can be obtained without obtaining a full minimal presentation of $S$.  
In fact, upon re-examining~\cite{rosalesApery,highembdim}, one can see the arguments and constructions therein as special cases of those we develop in Sections~\ref{sec:lowerbound} and~\ref{sec:intervalfamily}, albeit with much lengthier arguments and in less generality.  

The purposes of the present manuscript are twofold.  
\begin{itemize}
\item
In Section~\ref{sec:overview}, we provide a self-contained introduction to the machinery introduced in~\cite{kunzfaces3}, illustrating how it can be used to approach Question~\ref{q:mainquestion}.  Although the manuscript is the third in a sequence of geometry-centric papers~\cite{kunzfaces1,kunzfaces2,kunzfaces3}, the combinatorial methods for obtaining $\eta(S)$ do not actually rely on the geometry, and our overview of these methods in Section~\ref{sec:overview} is careful to avoid it.  

\item 
In the remaining sections, we utilize this new machinery to make considerable headway on Question~\ref{q:mainquestion}.  We present several families of numerical semigroups achieving a large range of values of $\eta$ for each $e$ and $m$, as well as bounds on the possible values of $\eta$.  

\end{itemize}

\begin{figure}[t]
\begin{center}
\includegraphics[width=5.0in]{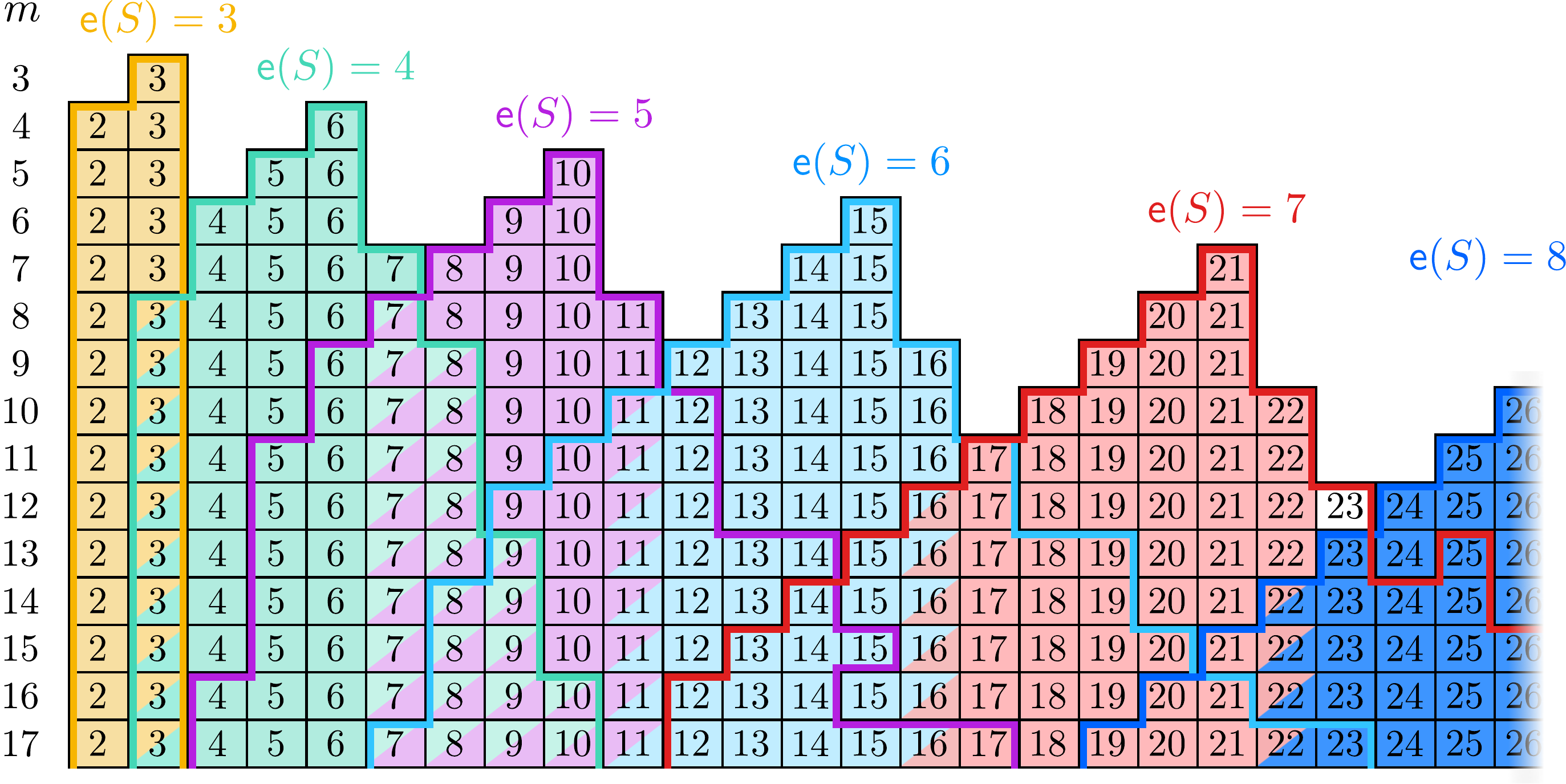}
\end{center}
\caption{The values of $\eta(S) \le 26$ attained for $\mathsf m(S) \le 17$, with those attained for each $\mathsf e(S) \le 8$ outlined.}
\label{fig:achieved}
\end{figure}

The diagram in Figure~\ref{fig:achieved} lists all values of $\eta \le 26$ achieved for $m \le 17$ and $e \le 8$, obtained computationally using algorithms from~\cite{wilfmultiplicity,kunzfaces3}.  Each row corresponds to a value of $m \le 17$, and each boxed number in that row is the value of $\eta$ achieved by some numerical semigroup with multiplicity $m$. Each bold colored edge demarcates the values of $\eta$ that are achieved by numerical semigroups with the labeled embedding dimension $e \le 8$, and the top box of each outlined region is $\eta = \binom{e}{2}$.  

Figure~\ref{fig:achieved} puts the results of this manuscript in context.  Shaded boxes indicate values of $\eta$ attained by families of numerical semigroups constructed in Theorems~\ref{t:intervalfamily}, \ref{t:extrabettifamily}, and~\ref{t:embdim4}, colored according to the value of $e$ used therein. For comparison, the results of~\cite{highembdim} characterize the first 3 rows for each embedding dimension.  Theorem~\ref{t:lowerbound} gives a lower bound $\eta \ge \binom{e}{2} - (m - e)$, which can be seen as the ``staircase'' each bold colored edge makes from $\eta = \binom{e}{2}$ down and to the left to $\eta = \binom{e-1}{2} + 1$.  We also prove in Theorem~\ref{t:upperbound} that if $e = m - 3$, then $\eta \le \binom{e}{2} + 1$, which characterizes the 4th row for each embedding dimension, and provide a streamlined proof of the upper bound given in~\cite{highembdim} for the first 3 rows.  
In Proposition~\ref{p:eta3} and Theorem~\ref{t:embdim4}, we identify additional families of numerical semigroups with $e = 4$ that achieve every value of $\eta$ outlined in the green edges.  
In~fact, we conjecture that these families achieve all possible values of~$\eta$ for every $m$ when $e = 4$, which we have verified computationally for $m \le 42$.  We~close with Section~\ref{sec:extras}, which contains several open problems, along with a proof that upon restricting to each $e \geq 4$, every column in Figure~\ref{fig:achieved} with $\eta \ge e - 1$ has only finitely many missing boxes.  

One additional consequence of our results pertains to the related question ``given a multiplicity $m$, what are the possible values of $\eta$?''  Only a narrow range of values now remains uncharacterized, namely those attained when 
$$
e + 3 < m < 2e
\qquad \text{and} \qquad
\textstyle\binom{e}{2} + 2 \le \eta \le \binom{e}{2} + (2e - m).
$$
The only such values in Figure~\ref{fig:achieved} are $\eta = 23$ for $m = 11$ and $m = 12$; the latter is achieved when $e = 7$, while the former is not achieved by any numerical semigroup. Indeed, $\eta = 23$ is achieved for each $m \ge 13$ with $e = 8$ by Theorem~\ref{t:intervalfamily}.

\section{An overview of nilsemigroups and outer Betti elements}
\label{sec:overview}

In this section, we provide a self-contained introduction to the machinery introduced in~\cite{kunzfaces3}, including Kunz nilsemigroups and outer Betti elements, with an emphasis on illustrating how this machinery can be used to approach Question~\ref{q:mainquestion}.  All definitions appearing here that are not in~\cite{kunzfaces3} can be found in the monographs~\cite{numericalappl,numerical}.  

Fix a numerical semigroup $S = \<n_1, \dots, n_k\>$.  
The \emph{embedding codimension} of $S$ is 
$$\mathsf r(S) = \mathsf m(S) - \mathsf e(S).$$
In particular, max embedding dimension numerical semigroups have embedding codimension 0.  
Letting $m = \mathsf m(S)$, the \emph{Apery set} of $S$ is the set
$$\Ap(S) := \{n \in S: n - m \notin S\}$$
of minimal elements of $S$ within each equivalence class modulo $m$.  Since $S$ is cofinite, we~are guaranteed $|\!\Ap(S)| = m$, and that $\Ap(S)$ contains exactly one element in each equivalence class modulo $m$.  As such, we often write
$$\Ap(S) = \{a_0, a_1, \ldots, a_{m-1}\}$$
with each $a_i \equiv i \bmod m$, and view the subscripts as elements of $\ZZ_m$ (e.g. $a_{-1} = a_{m-1}$).  

Recall that a \emph{factorization} of an element $n \in S$ is an expression
$$n = z_1n_1 + \cdots + z_kn_k$$
of $n$ as a sum of atoms of $S$, which we often encode as a $k$-tuple $z = (z_1, \ldots, z_k) \in \ZZ_{\ge 0}^k$, and its \emph{length} is $z_1 + \cdots + z_k$.  
The \emph{factorization homomorphism}
$$\begin{array}{r@{}c@{}l}
\varphi_S:\ZZ_{\ge 0}^k &{}\longrightarrow{}& S \\
z &{}\longmapsto{}& z_1n_1 + \cdots + z_kn_k
\end{array}$$
is the additive semigroup homomorphism that sends each $k$-tuple $z = (z_1, \ldots, z_k)$ to the element of $S$ that $z$ is a factorization of.  Under this notation, the preimage $\varphi_S^{-1}(n) = \mathsf Z_S(n)$ is the \emph{set of factorizations} of $n \in S$.  
The \emph{kernel} of $\varphi_S$, denoted $\til = \ker\varphi_S$, relates $z \sim  z'$ whenever $\varphi_S(z) = \varphi_S(z')$, in which case we call the pair $(z,z')$ a \emph{trade} or \emph{relation}.  
The kernel is a \emph{congruence}, i.e., an equivalence relation satisfying $z + z'' \sim z' + z''$ whenever $z \sim z'$ and $z'' \in \ZZ_{\ge 0}^k$.  
A subset $\rho \subseteq \ker\varphi_S$ is called a \emph{presentation} for $S$ if the intersection of all congruences containing $\rho$ is $\ker\varphi_S$.  A presentation $\rho$ of $S$ is \emph{minimal} if no proper subset of $\rho$ is a presentation for $S$.  It is known that any two minimal presentations of $S$ have the same cardinality, which we denote $\eta(S) = |\rho|$.  
The \emph{Betti elements} of $S$ are those in the set
$$\Betti(S) := \{\varphi_S(z) : (z, z') \in \rho\},$$
where $\rho$ is any minimal presentation of $S$; the set $\Betti(S)$ is independent of the choice of $\rho$.  
The \emph{factorization graph} $\nabla_n$ of an element $n \in S$ has vertex set $\mathsf Z_S(n)$ and distinct vertices $z, z' \in \mathsf Z_S(n)$ are connected by an edge whenever $z_i > 0$ and $z_i' > 0$ for some $i$.  It is known that $n \in \Betti(S)$ if and only if $\nabla_n$ is disconnected, and in fact the number of relations $(z, z') \in \rho$ for which $n = \varphi(z)$ is one less than the number of connected components of~$\nabla_n$.  

Before seeing an example, we give one more definition that will be needed for several proofs in later sections.  
Given a second numerical semigroup $S' = \<n_1', \dots, n_\ell'\>$ and non-atoms $a \in S$ and $a' \in S'$ with $\gcd(a,a') = 1$, the \emph{gluing} of $S$ and $S'$ by $a$ and $a'$ is
$$T = a'S + aS' = \<a' n_1, \ldots, a' n_k, an_1', \ldots, an_\ell'\>,$$
for which is it known \cite{delormegluings,completeintersection} that $\mathsf e(T) = \mathsf e(S) + \mathsf e(S')$, $\mathsf m(T) = \max(a'\mathsf m(S), a\mathsf m(S'))$, and $\eta(T)=\eta(S)+\eta(S')+1$.

\begin{example}\label{e:firstexamples}
The numerical semigroup $S_1 = \<6, 7, 8, 9, 10, 11\>$ has max embedding dimension, so every nonzero element in its Ap\'ery set
$$\Ap(S_1) = \{0,7,8,9,10,11\}$$
is an atom of $S_1$.  On the other hand, $S_2 = \<8,9,28,14,15\>$ has Ap\'ery set
$$\Ap(S_2) = \{0,9,18,27,28,29,14,15\},$$
containing the non-atoms 18 and 27, both multiples of the atom 9, and $29 = 14 + 15$.  In general, elements of the Ap\'ery set are precisely those which have no factorizations involving the multiplicity.  
Also, $S_3 = \<10,22,23,24\>$ has Ap\'ery set 
$$\Ap(S_3) = \{0,71,22,23,24,45,46,47,48,69\}$$
and the trade $(0,0,2,0) \sim (0,1,0,1)$ between factorizations of $46 \in \Ap(S_3)$. In fact this trade lies in every minimal presentation of $S_3$.  
\end{example}

Recall that a nilsemigroup is a semigroup with a universally absorbing element $\infty$, called the \emph{nil}. Let $(N,+)$ be a nilsemigroup that is finite, has an identity $0 \in N$, and is \emph{partly cancellative}: $a + b = a + c \ne \infty$ implies $b = c$ for all $b, c \in N$.
Like numerical semigroups, any finite partly cancellative nilsemigroup has a unique minimal generating set.  We~write $\mathsf m(N) = |N| - 1$ for the number of non-nil elements and $\mathsf e(N)$ for one more than the number of minimal generators, which are also called \emph{atoms}.  

The \emph{Kunz nilsemigroup} $N$ of a numerical semigroup $S$ is obtained from $S/\til$, where $\til$ is the congruence that relates $a \sim b$ whenever $a = b$ or $a, b \notin \Ap(S)$ (the set $S \setminus \Ap(S)$ comprises the nil of $S/\til$), by replacing each non-nil element with its equivalence class in $\ZZ_m$.  This ensures $\mathsf e(N) = \mathsf e(S)$ and $\mathsf m(N) = \mathsf m(S)$, as the minimal generators of $N$ are the minimal generators of $S$ distinct from the multiplicity.  

A finite partly cancellative nilsemigroup $N$ may be visualized by examining the \emph{divisibility poset} $P$ of non-nil elements, wherein $b \preceq c$ when $c = a + b$ for some $a \in N$.  Partial cancellativity ensures $c$ covers $b$ when $c = a + b$ for some atom $a$, so in fact the additive structure of $N$ can be recovered from the poset structure of $P$.  If $N$ is the Kunz nilsemigroup of a numerical semigroup $S$, we call $P$ the \emph{Kunz poset} of $S$.  

\begin{example}\label{e:kunzposets}
The Kunz posets of the numerical semigroups $S_1, S_2$, and $S_3$ from Example~\ref{e:firstexamples} are depicted in Figures~\ref{fig:examples1} and~\ref{fig:examples2}, with one black dot for each non-nil element.  The dashed edges and red dots in each depiction will be addressed in Examples~\ref{e:nilsemigroupminpres} and~\ref{e:outerbettis} once the necessary definitions have been discussed.  

Each element covering 0 is a nilsemigroup atom, and the edges throughout each depiction are colored to reflect the fact that each cover relation results from adding a nilsemigroup atom.  Considering the nilsemigroup $N_2$ of $S_2$, for instance, the 3 nonzero non-atoms are 2 and 3, both of which are multiples of $1 \in N_2$, and $5 = 6 + 7 \in N_2$.  This perfectly encodes the additive structure of the elements of $\Ap(S_2)$ discussed in Example~\ref{e:firstexamples}.  Moreover, one can see in the depiction of the Kunz nilsemigroup $N_3$ of $S_3$ in Figure~\ref{fig:examples2} that $a_6 \in \Ap(S_3)$ has two distinct factorizations $a_6 = a_2 + a_4 = 2a_3$, which constitute a trade $(0,1,0,1) \sim (0,0,2,0)$.  
\end{example}

\begin{figure}[t]
\begin{center}
\includegraphics[height=1.2in]{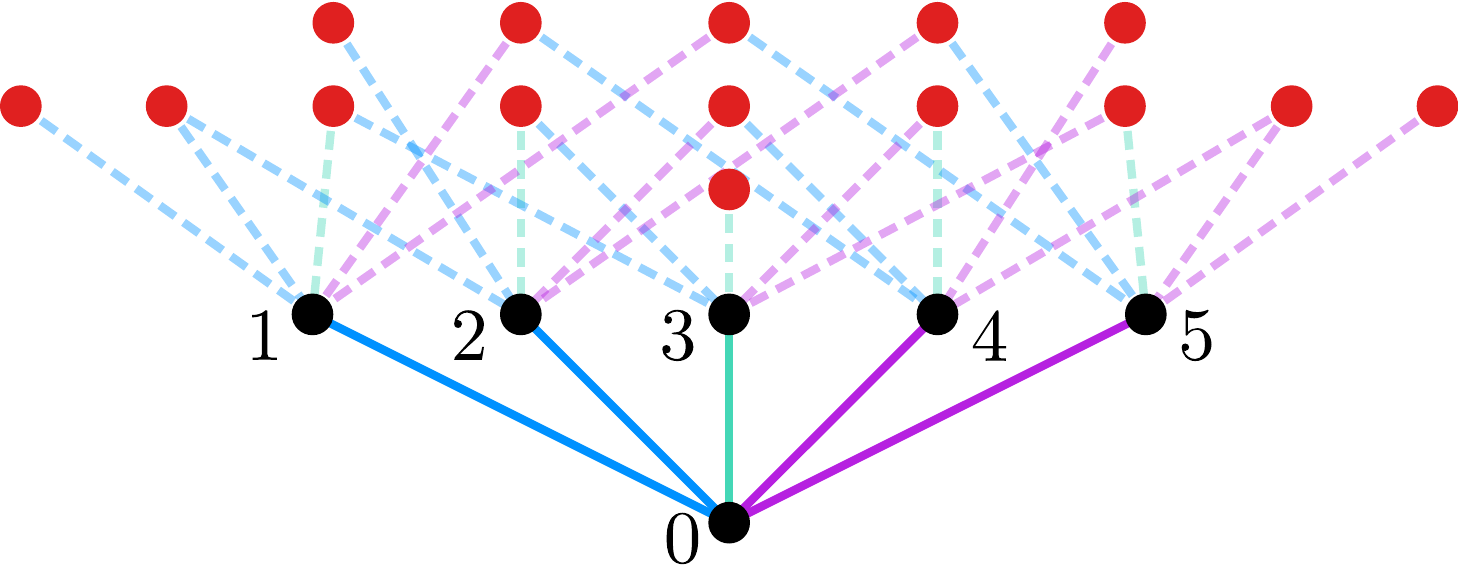}
\hspace{3em}
\includegraphics[height=1.7in]{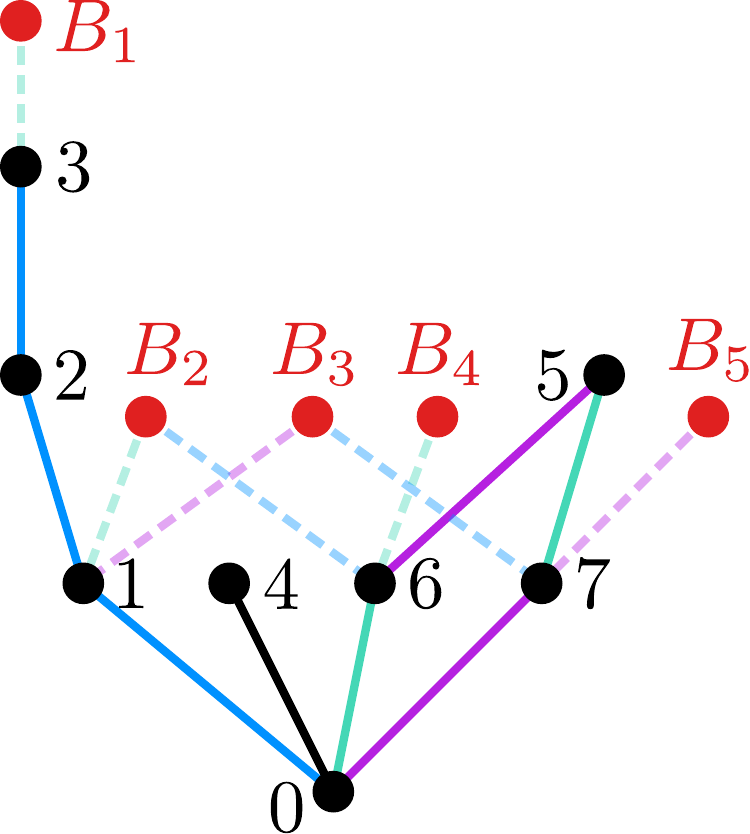}
\end{center}
\caption{The Kunz posets of $S_1$ (left) and $S_2$ (right) from Example~\ref{e:kunzposets}, with outer Betti elements depicted as in Example~\ref{e:outerbettis}.}
\label{fig:examples1}
\end{figure}

\begin{figure}[t]
\begin{center}
\includegraphics[height=1.5in]{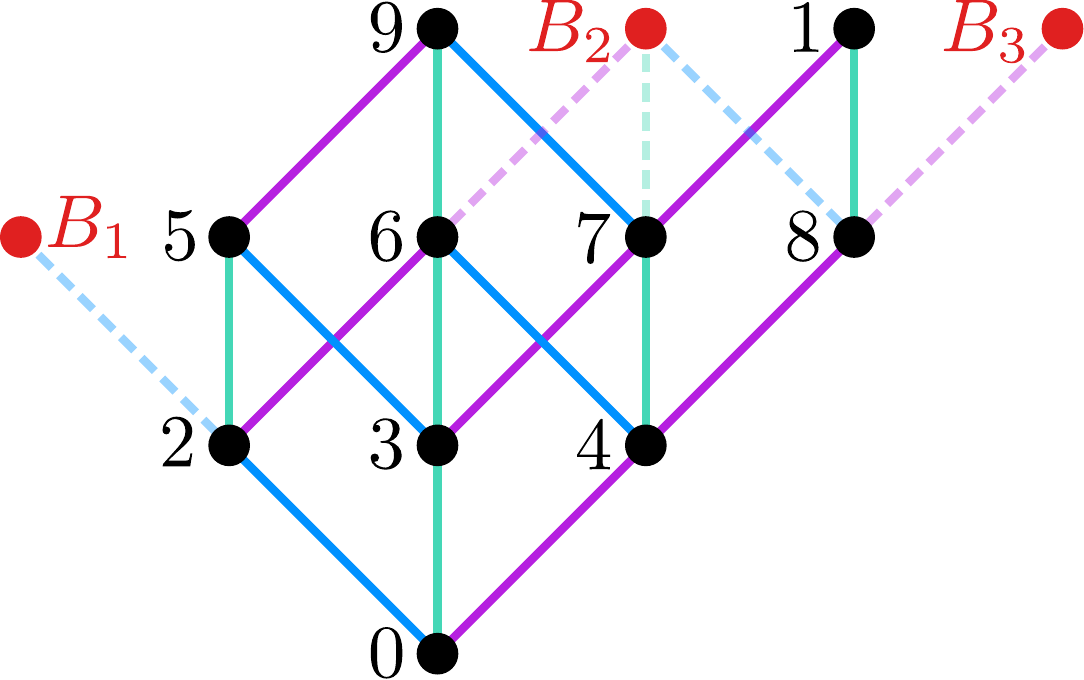}
\hspace{3em}
\includegraphics[height=1.5in]{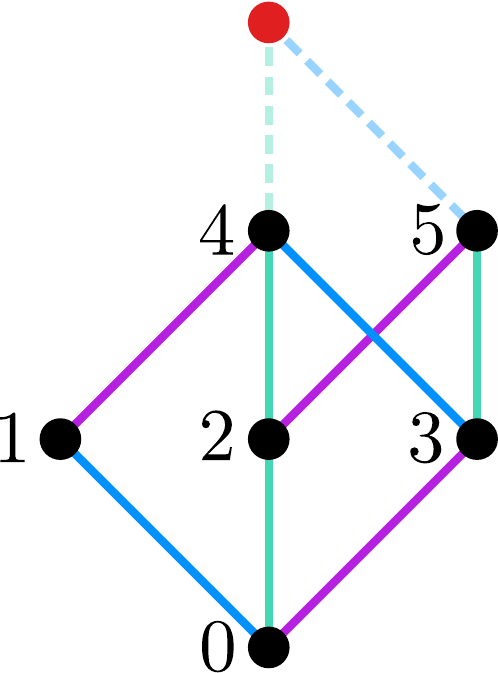}
\end{center}
\caption{The Kunz posets of the semigroups $S_3 = \<10,22,23,24\>$ (left) from Example~\ref{e:kunzposets} and $S_4 = \<6,7,8,9\>$ (right) from Example~\ref{e:lemma56}.}
\label{fig:examples2}
\end{figure}


Given a finite partly cancellative nilsemigroup $N$ with $\mathsf e(N) = k$, one can analogously define the factorization homomorphism $\varphi_N:\ZZ_{\ge 0}^{k-1} \to N$, with $\mathsf Z_N(p) = \varphi_N^{-1}(p)$ for each $p \in N$.  Partial cancellativity ensures $|\mathsf Z_N(p)| < \infty$ unless $p = \infty$.  
If $N$ is the Kunz nilsemigroup of a numerical semigroup $S$, then for each $a_i \in \Ap(S)$, omitting the first coordinate of each factorization in $\mathsf Z_S(a_i)$ yields $\mathsf Z_N(i)$, and $\mathsf Z_N(\infty)$ contains all remaining elements of $\ZZ_{\ge 0}^{k-1}$.  

We say $\rho \subset \ker\varphi_N$ is a \emph{minimal presentation} if it is obtained from a minimal generating set of $\ker\varphi_N$ by omitting any trade $z \sim z'$ with $\varphi_N(z) = \infty$.  Likewise, 
$$\Betti(N) := \{\varphi_N(z) : (z, z') \in \rho\}.$$
The omission of trades occurring at $\infty$ is a slight departure from the ``usual'' definition of a minimal presentation from semigroup theory, but is more natural if $N$ is the Kunz nilsemigroup of a numerical semigroup $S$.  A minimal presentation for $N$ can be obtained from a minimal presentation $\rho$ for $S$ by (i) omitting any trade $(z, z')$ with $\varphi_S(z) \notin \Ap(S)$ and (ii) omitting the leading 0 entry from both factorizations in all remaining trades.  In fact, $$\Betti(N) = \{i : a_i \in \Ap(S) \cap \Betti(S)\}.$$
There is also a partial converse.  Fix a minimal presentation $\rho$ for $S$, partitioned as $\rho = \rho' \cup \rho''$ where $(z, z') \in \rho''$ whenever $\varphi_S(z) \in \Ap(S)$.  Let $\rho'''$ be a collection of trades for $S$ obtained from a minimal presentation for $N$ by prepending a 0 entry to both factorizations in each trade.  Then $\rho' \cup \rho'''$ is also minimal presentation for $S$.

\begin{example}\label{e:nilsemigroupminpres}
The only numerical semigroup in Example~\ref{e:firstexamples} whose Kunz nilsemigroup has nonempty minimal presentation is $S_3$, and $\{((1,0,1), (0,2,0))\}$ is in fact the only possible minimal presentation for its Kunz nilsemigroup $N_3$.  
Note the distinction between the above trade and the one for $S_3$ given at the end of Example~\ref{e:kunzposets}:\ since $S_3$ has one additional generator, namely $\mathsf m(S) = 10$, each factorization has one additional entry.  
For comparison, $\eta(S_3) = 4$ and 
$$\Betti(S_3) = \{44, 46, 70, 72\},$$
with $\Ap(S_3) \cap \Betti(S_3) = \{46\}$.  
One possible minimal presentation $\rho$ of $S_3$ has trades
$$
\begin{array}{@{}l@{\,\,\,\,}l@{}}
(0,2,0,0) \sim (2,0,0,1), & (0,1,0,1) \sim (0,0,2,0), \\
(0,0,2,1) \sim (7,0,0,0), & (0,0,0,3) \sim (5,1,0,0)
\end{array}
$$
each occurring at the corresponding Betti element.  
It is at this point that we can begin to see the role the red dots and dashed edges play in Figures~\ref{fig:examples1} and~\ref{fig:examples2}:\ these are the locations of the trades in $\rho$ that occur at Betti elements outside the Ap\'ery set.  Indeed, each red dot is labeled with the equivalence class of some $b \in \Betti(S_3) \setminus \Ap(S_3)$ modulo $\mathsf m(S_3) = 10$, and the first factorization in each such trade above indicates, as we will see below, the ``factorization'' of the corresponding red dot.  
\end{example}

We are now ready to define outer Betti elements.  First, the \emph{support} of a factorization $z \in \ZZ_{\ge 0}^k$ and a subset $Z \subset \ZZ_{\ge 0}^k$ are given by
$$
\supp(z) = \{i : z_i > 0\}
\qquad \text{and} \qquad
\supp(Z) = \{i : z_i' > 0 \text{ for some } z' \in Z\},
$$
respectively.  Let $\nabla_Z$ denote the graph whose vertex set is $Z$ wherein distinct vertices $z, z' \in Z$ are connected by an edge whenever $\supp(z) \cap \supp(z')$ is nonempty, and for each $i \in \supp(Z)$, define 
$$Z - e_i = \{ z - e_i :  z \in Z \text{ with } i \in \supp(z) \},$$
where $e_i$ is the $i$-th standard basis vector.  

Now, an \emph{outer Betti element} of a finite partly cancellative nilsemigroup $N$ is a subset $B \subset Z_N(\infty)$  such that
\begin{enumerate}[(i)]
\item 
for every $i \in \supp(B)$, we have $B - e_i = \mathsf Z_N(p)$ for some $p \in N \setminus \{\infty\}$, and

\item 
the graph $\triangledown_B$ is connected.
\end{enumerate}
We denote by 
$$b(N) = \text{number of outer Betti elements of } N,$$
and $\eta(N) = b(N) + |\rho|$, where $\rho$ is any minimal presentation of $N$.  

Before examining outer Betti elements in more detail, we present the following consolidation of the main results of~\cite[Section~5]{kunzfaces3} pertaining to outer Betti elements and minimal presentations.  

\begin{thm}\label{t:kunzfaces3}
If $\rho$ is a minimal presentation for the Kunz nilsemigroup $N$ of a numerical semigroup $S$, then $\rho' \cup \rho''$ is a minimal presentation for $S$, where:
\begin{enumerate}[(i)]
\item
$\rho'$ contains one trade $(z, z')$ for each outer Betti element $B$ of $N$, where $z$ is obtained from a factorization in $B$ by prepending a 0, and $z'$ is any factorization of $\varphi_S(z)$ with positive first coordinate; and

\item 
$\rho''$ is obtained from $\rho$ by prepending a 0 to both factorization of each trade.  

\end{enumerate}
In particular, $\eta(S) = \eta(N) = b(N) + |\rho|$.
\end{thm}

\begin{example}\label{e:outerbettis}
Any numerical semigroup $S$ with an \emph{Ap\'ery set of unique expression}, meaning that each Ap\'ery set element has exactly one factorization, has Kunz nilsemigroup $N$ whose outer Betti elements are singletons, and each contains a factorization of $\infty$ that is minimal with respect to the componentwise partial order. As such, $\eta(S) = b(N)$ by Theorem~\ref{t:kunzfaces3}.  
This is the case for the numerical semigroups $S_1$ and $S_2$ from Example~\ref{e:firstexamples}, but not $S_3$.  

For the Kunz nilsemigroup $N_1$ of $S_1$, 
$$\mathsf Z_{N_1}(\infty) = \{z \in \ZZ_{\ge 0}^5 : z_1 + \cdots + z_5 \ge 2\},$$
so there are $\binom{6}{2}$ outer Betti elements, each containing a length 2 factorization.  
This generalizes to a known result that any max embedding dimension numerical semigroup $S$ with $\mathsf m(S) = m$ has $\eta(S) = \binom{m}{2}$, with one trade for each length 2 factorization not involving $m$.  

For the Kunz nilsemigroup $N_2$ of $S_2$, we have that $\mathsf Z_{N_2}(\infty)$ has $\eta(S) = 9$ factorizations that are minimal with respect to the component-wise partial order:\ the 5 depicted in Figure~\ref{fig:examples1}, and one for each factorization of the form $e_2 + e_i$ for $i \in [1,4]$.  Note that a finite partly cancellative nilsemigroup with 4 atoms and no other nonzero non-nil elements would have 10 outer Betti elements:\ 
\begin{itemize}
\item
one is $\{e_3 + e_4\}$, which is not an outer Betti element of $N_2$ since $6 + 7 = 5 \in N_2$;

\item
one is $\{2e_1\}$, which is not an outer Betti element of $N_2$, although $\{4e_1\}$ is an outer Betti element with identical support; and

\item
the other 8 are identical to those of $N_2$.  

\end{itemize}
These ideas are utilized in constructing the family of semigroups in Theorem~\ref{t:intervalfamily}.  

Unlike $S_1$ and $S_2$, the semigroup $S_3$ does not have an Ap\'ery set of unique expression.  Indeed, the Kunz nilsemigroup $N_3$ of $S_3$ has 3 singleton outer Betti elements, along with the outer Betti element 
$$B_3 = \{(0,2,1), (1,0,2)\}$$
which is non-singleton since it lies above $6 \in N_3$, which has 2 factorizations.  Intuitively, the factorization graph $\nabla_{70}$ has an edge between $(0,0,2,1)$ and $(0,1,0,2)$, the trade between which occurs at $46$.  This is the motivation for requirement~(ii) in the definition of outer Betti element:\ any minimal factorizations of the nil that are connected by a trade at a non-nil element cannot yield more than one trade under Theorem~\ref{t:kunzfaces3}.  
More generally, although it need not be obvious from definitions, Theorem~\ref{t:kunzfaces3} implies that prepending a 0 to any two factorizations in a given outer Betti element $B$ yields two factorizations of the same element of $S$.  
\end{example}

\begin{example}\label{e:lemma56}
The numerical semigroup $S_4 = \<6,7,8,9\>$, and its Kunz nilsemigroup $N$ depicted in Figure~\ref{fig:examples2}, illustrate the subtleties of part~(i) in the definition of outer Betti elements.  On the one hand, any factorization in an outer Betti element $B$ must be a minimal element of $\mathsf Z_N(\infty)$.  However, the converse need not hold:\ $(0,2,1) \in \mathsf Z_N(\infty)$ is minimal, but it does not lie in any outer Betti element, as the trade $(0,2,0) \sim (1,0,1)$ occurring at $4 \in N$ connects it via an edge to $(1,0,2)$, which is not minimal in $\mathsf Z_N(\infty)$.  Indeed, we see $25 = 7 + 9 + 9$ is not a Betti element of $S_4$ since its 3 factorizations form a connected graph $\nabla_{25}$.  

There is also an algorithmic way to compute the outer Betti elements of a given nilsemigroup from the set of minimal factorizations of the nil.  Build a graph $G$ whose vertex set $Z$ is comprised of the minimal elements of $\mathsf Z_N(\infty)$, and include an edge between $z, z' \in Z$ whenever $z - e_i, z' - e_i \in \mathsf Z_N(p)$ for some $i \in \supp(z) \cap \supp(z')$ and non-nil $p \in N$.  By \cite[Lemma~5.6]{kunzfaces3}, each outer Betti element will be a connected component of $G$, so one simply needs to compute the connected components of this graph and check which satisfy condition~(i).  In particular, any connected component~$B$ of $G$ has $\nabla_B$ connected, and for each $i \in \supp(B)$, $B - e_i \subseteq \mathsf Z_N(p)$ for some non-nil $p$, so $B$ is an outer Betti element of $N$ if and only if equality holds for each $i$.  
\end{example}

\begin{example}\label{e:outerbettinotbetti}
Note that each outer Betti element corresponds to an element of $\rho$, not an element of $\Betti(S)$.  In particular, two outer Betti elements can correspond to relations under Theorem~\ref{t:kunzfaces3} that occur at the same element of $S$.  With~$S_1$ as an example, the outer Betti elements 
$$
B = \{(0,1,1,0,0)\}
\qquad \text{and} \qquad
B' = \{(1,0,0,1,0)\}
$$
each yield a relation at $17 = 8 + 9 = 7 + 10 \in S_1$.  More generally, the outer Betti elements of $S_1$ that yield relations at the same element of $S_1$ are depicted above/below each other in Figure~\ref{fig:examples1}.  
On the other hand, $S = \<6,13,8,9,10,11\>$ has identical Kunz nilsemigroup to $S_1$, but $B$ and $B'$ yield relations at $17 = 8 + 9$ and $23 = 13 + 10$, respectively, so one cannot determine from the Kunz nilsemigroup alone which outer Betti elements yield relations at the same numerical semigroup element.  This is an advantage when examining minimal presentation cardinality using Kunz nilsemigroup machinery, as it eliminates potential casework.  
\end{example}

\begin{remark}\label{r:groenberbases}
There is a connection between Ap\'ery sets and Gr\"obner bases of polynomial ideals that is relevant here.  
The kernel $I_S$ of the map $\QQ[x_1, \ldots, x_k] \to \QQ[t]$ given by $x_i \mapsto t^{n_i}$ is known as the defining toric ideal of $S = \<n_1, \ldots, n_k\>$.  A minimal binomial generating set for $I_S$ has one binomial for each trade in a minimal presentation of $S$, and such a minimal generating set can be computed using Gr\"obner bases (see~\cite{grobpoly} for background on Gr\"obner bases of toric ideals).  
The ideal $J = I_S + \<x_1\>$ has been utilized to study homological properties of $I_S$ and to obtain an algorithm for computing $\Ap(S)$ that utilizes Gr\"obner bases~\cite{compapery,shortresolutionalg}.  In this context, one may obtain a minimal generating set for $J$ consisting of $x_1$ and one binomial for each trade in a minimal presentation of $N$.  In fact, under certain term orders, the initial ideal $M$ of~$I_S$ has one monomial generator for each outer Betti element of $N$, and each monomial outside of $M$ corresponds to a factorization of a distinct element of $\Ap(S)$.  
\end{remark}

\section{A lower bound on minimal presentation cardinality}
\label{sec:lowerbound}

\begin{notation}
Unless otherwise stated, in the remainder of the paper:\ $m, e, r, \eta \in \ZZ_{\ge 0}$ are fixed with $m \ge e \ge 3$ and $m = e + r$; $S$~denotes a numerical semigroup with $\mathsf m(S) = m$, $\mathsf e(S) = e$, $\mathsf r(S) = r$, and $\eta(S) = \eta$; and $N$~denotes a finite partly cancellative nilsemigroup with $m$ non-nil elements, $e-1$ atoms and $\eta(N) = \eta$.  
\end{notation}

In this section, we prove $\eta \ge \binom{e}{2} - r$ for any numerical semigroup, a lower bound we will demonstrate is sharp for $r < e$ in Theorem~\ref{t:intervalfamily}.  This lower bound coincides with the one for $r \in [0,2]$ given in~\cite{highembdim}.

\begin{defn}\label{d:maximalelement}
An element $p\in N\setminus\{\infty\}$ is called \emph{maximal} if $p+p' = \infty$ for any nonzero $p' \in N$. The \emph{quotient} of $N$ by a maximal element $p$, which we denote by $N/p$, is the quotient nilsemigroup $N/\til$ where $\til$ is the congruence whose only nontrivial relation is $p \sim\infty$.  
In particular, $\mathsf Z_{N/p}(\infty) = \mathsf Z_N(\infty) \cup \mathsf Z_N(p)$, while $\mathsf Z_{N/p}(p') = \mathsf Z_N(p')$ for each $p' \in N \setminus \{p, \infty\}$.  
\end{defn}

\begin{lemma}\label{l:lowerbound}
Let $p$ be a maximal element of $N$. If there are $k$ minimal relations occurring at $p$, then $b(N) + k + 1 - b(N/p)$ equals the number of outer Betti elements of $N$ divisible by $p$.  In particular, $b(N/p) - 1 \le b(N) + k$.  
\end{lemma}

\begin{proof}
Fix an outer Betti element $B$ of $N$.  If $B$ is divisible by $p$ (that is, $B - e_j = \mathsf Z_N(p)$ for some $j \in \supp(B)$), then no factorization in $B$ can appear in an outer Betti element of $N/p$ since $B - e_j \subset \mathsf Z_{N/p}(\infty)$.  If, on the other hand, $B$ is not divisible by~$p$, then $B$ is also an outer Betti element of $N/p$, as the factorizations in $B - e_j$ are unaffected by the quotient for each $j \in \supp(B)$.  Each outer Betti element of $N/p$ thus either coincides with an outer Betti element of $N$ or consists of factorizations in $\mathsf Z_N(p)$.  

Let $B_1, \ldots, B_n$ denote the outer Betti elements of $N/p$ that are contained in $\mathsf Z_N(p)$.
Since outer Betti elements have connected factorization graphs, and since for each $j \in \supp(B_i)$, $B_i - e_j = \mathsf Z_N(p')$ for some $p' \in N$, the connected components of $\nabla_p$ in $N$ must be precisely $B_1, \ldots, B_n$.  This implies there are $n - 1 = k$ relations occurring at $p$ in $N$, so we obtain
\[b(N/p) - 1 \le b(N) + n - 1 = b(N) + k\]
thereby proving our claim.
\end{proof}

\begin{thm}\label{t:lowerbound}
For any numerical semigroup $S$, we have $\eta \ge \binom{e}{2} - r$.
\end{thm}

\begin{proof}
Let $N$ denote the Kunz nilsemigroup of $S$, and let $p_1, \ldots, p_r \in N$ denote the nonzero non-nil non-atoms of $N$, ordered so that each $p_i$ is maximal with respect to divisibility among $p_i, \ldots, p_r$.  Define partly cancellative nilsemigroups $N_0, \ldots, N_r$ so that $N_0 = N$ and $N_i = (N_{i-1})/p_i$ for each $i \ge 1$.  Letting $k_i$ be the number of relations occurring at $p_i$ in $N_{i-1}$ (which coincides with the number of relations occurring at $p_i$ in $N$), Lemma~\ref{l:lowerbound} implies 
\[
\eta(N) = b(N_0) + \sum_{i = 1}^r k_i \ge b(N_1) - 1 + \sum_{i = 2}^r k_i \ge \cdots \ge b(N_r) - r = \binom{e}{2} - r
\]
since $N_r$ consists of $0$, $\infty$, and the atoms of $N$.  Theorem~\ref{t:kunzfaces3} completes the proof.  
\end{proof}

\begin{remark}\label{r:generalnilsemigroups}
Lemma~\ref{l:lowerbound} illustrates another advantage of reformulating questions about minimal presentations in terms of partly cancellative nilsemigroups. 
Certain operations that can be defined on nilsemigroups--in this case, quotients by maximal elements--are not possible if one is confined to Kunz posets (or numerical semigroups, for that matter).  We will utilize this generality again in Theorem~\ref{t:upperbound} to streamline the proof of upper bounds on $\eta$ for numerical semigroups with small codimension.  
Additionally, as we will see in Sections~\ref{sec:intervalfamily} and~\ref{sec:embdim4}, shedding unnecessary information about the original numerical semigroup in favor of its Kunz nilsemigroup can help streamline arguments that a given numerical semigroup $S$ has a claimed value $\eta(S)$.  
\end{remark}

\section{An interval of attainable minimal presentation cardinalities}
\label{sec:intervalfamily}



In this section, we construct a family of numerical semigroups attaining each minimal presentation cardinality in the interval $[\binom{e}{2}-\min(r, e-1), \binom{e}{2}]$.  This family simultaneously generalizes those in~\cite{rosalesApery} and~\cite{highembdim} using the machinery of Kunz nilsemigroups.  

\begin{example}\label{e:intervalfamily}
In \cite{rosalesApery}, the family of numerical semigroups 
\[
S = \< m, m+1, (r+2)m+(r+2), \dots, (r+2)m+(m-1) \>
\]
is introduced to exhibit a numerical semigroup $S$ with $\eta(S) = \binom{e}{2}$ and with any multiplicity $m \ge e$.  The Kunz poset of the above numerical semigroup is nearly identical to the one depicted in Figure~\ref{fig:intervalfamily}, except that each label is replaced with its negation modulo $m$.  
Intuitively, this construction ensures that, just as for max embedding dimension numerical semigroups, there is one outer Betti element for each support set of cardinality at most 2.  Extending to the family in Theorem~\ref{t:intervalfamily}, additional non-atoms are carefully placed to each eliminate one outer Betti element without creating any additional ones.  
\end{example}


\begin{figure}[t]
\begin{center}
\includegraphics[height=1.5in]{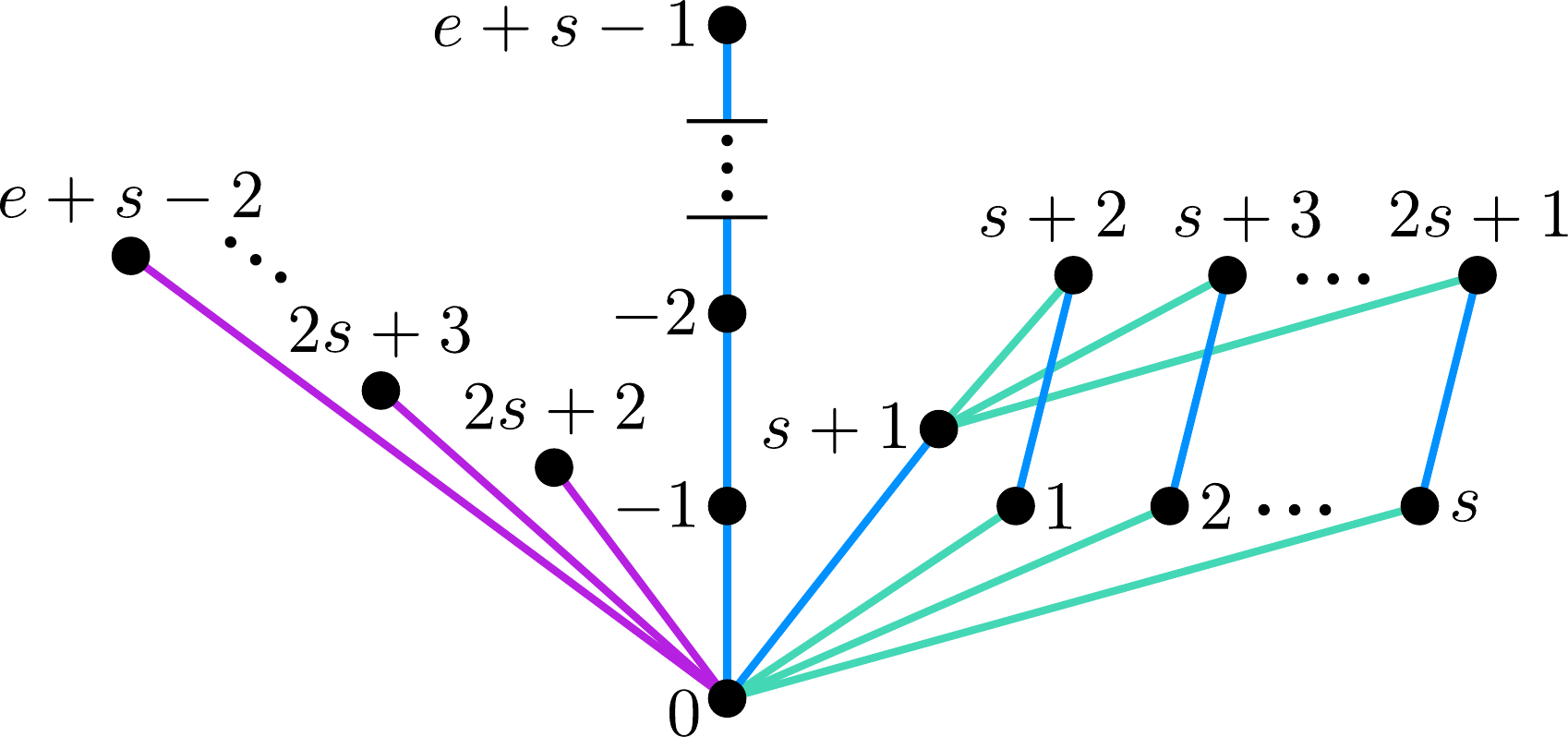}
\hspace{3em}
\includegraphics[height=1.5in]{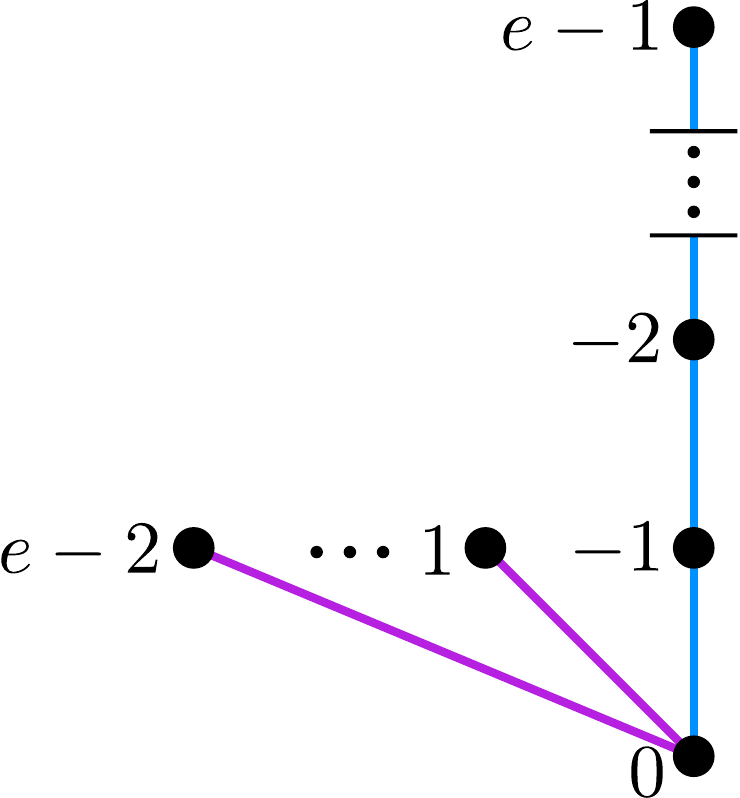}
\end{center}
\caption{The Kunz poset structure of $S$ with $\eta(S) = \binom{e}{2} - s$ from Theorem~\ref{t:intervalfamily} (left) and the case $s = 0$ from Example~\ref{e:intervalfamily} (right).}
\label{fig:intervalfamily}
\end{figure}

\begin{thm}\label{t:intervalfamily}
If $e \ge 4$ and $0 \le s \le \min(e-2,r)$, then there exists a numerical semigroup of embedding dimension $e$ and multiplicity $m = r + e$ such that
\[\eta(S) = \binom{e}{2}-s.\]
In particular, the lower bound in Theorem~$\ref{t:lowerbound}$ is sharp if $r \le e - 1$.  
\end{thm}

\begin{proof}
If $s = e - 2$, then $m = r + e \ge 2e - 2$.  Since $T = \<e-1, e, \ldots, 2e - 3\>$ has max embedding dimension, $\eta(T) = \binom{e-1}{2}$.  As~such, for any prime $q > m$, we see $S = qT + m\ZZ_{\ge 0}$ is a valid gluing since $m \in T$ is not a minimal generator, and
$$\eta(S) = \binom{e-1}{2} + 1 = \binom{e}{2} - (e - 2).$$
We now turn to the case where $s \le \min(e-3, r)$.  

Let $I$ denote the interval $[2s+2, e+s-2]$, and consider
$$
S = \<m, 4m - 1, (2r - 2s + 3)m + (s + 1), (2r - 2s + 4)m + i, (4r - 4s + 5)m + j : i \in [1, s], j \in I \>.
$$
In what follows, we will prove $S$ has the Kunz poset depicted in Figure~\ref{fig:intervalfamily} and identify its Ap\'ery set elements $a_i$ for $i \in \ZZ_m$.  
First, let $a_0 = 0$ and $a_i = n_i$ denote the generator of $S$ with $n_i \equiv i \bmod m$.  
Since $s \le e - 3$, $|I| = e - s - 3 \ge 0$, and since $s \le r$, each $n_i > 3m$.  
Letting 
$$
a_{i+s+1} = n_i + n_{s+1} = (4r - 4s + 7)m + (i+s+1)
\qquad \text{for each} \qquad
i \in [1, s],
$$
we see 
$$a_{e+s-2} = \begin{cases}
(4r - 4s + 7)m + (e + s - 2) & \text{if } |I| = 0; \\
(4r - 4s + 5)m + (e + s - 2) & \text{if } |I| > 0.
\end{cases}$$
This means that
$$a_{-j} = \ell n_{-1} = \ell(4m - 1) \in \Ap(S)
\qquad \text{for each} \qquad
\ell \in [0, r - s + 1],$$
as each is exceeded by any other sum of two $n_i$'s, but 
$$(r - s + 2)n_{-1} = (4r - 4s + 9)m + (e + s - 2) > a_{e+s-2}.$$
We claim $\Ap(S) = \{a_i : i \in \ZZ_m\}$, and that each $a_i$ has a unique factorization.  Indeed, any non-negative integer combination of 2 or more generators of $S$ that does not exceed $a_{2s+1} = \max \Ap(S)$ cannot involve $n_i$ with $i \in I$, and cannot involve more than one $n_j$ with $j \in [1,s]$, so one can then check $n_i + n_{-1} > n_{i-1}$ for each $i \in [1,s+1]$ and 
$$2n_{s+1} = (4r - 4s + 6)m + (2s + 2) > a_{2s+2} = (4r - 4s + 5)m + (2s + 2).$$

Having now proven $S$ has the Kunz poset depicted in Figure~\ref{fig:intervalfamily}, it remains to count outer Betti elements, which in this case each consist of a single factorization in $\mathsf Z_N(\infty)$ that is minimal with respect to the coordinate-wise partial order.  There are $s + 1$ elements of $\Ap(S)$ with a factorization of length 2, so the remaining $\binom{e}{2} - (s + 1)$ length 2 factorizations are minimal elements of $\mathsf Z_N(\infty)$.  Any other minimal $z \in \mathsf Z_N(\infty)$ has length at least 3, but cannot contain more than one of any generator except $n_{-1}$ since (i) $2n_i \notin \Ap(S)$ for each $n_i \ne n_{-1}$ and (ii) among any 3 distinct generators, there are 2 whose sum lies outside of $\Ap(S)$.  As such, the only minimal $z$ with length at least~3 is a multiple of $n_{-1}$, so
$$\eta(S) = 1 + \binom{e}{2} - (s + 1) = \binom{e}{2} - s$$
by Theorem~\ref{t:kunzfaces3}.  
\end{proof}

\section{A partial upper bound on minimal presentation cardinality}
\label{sec:upperbound}



In this section, we present a sharp upper bound on minimal presentation cardinality of numerical semigroups of embedding codimension at most $3$.  

\begin{remark}\label{r:upperboundcite}
For embedding codimension at most 2, the bounds we present here also appeared in~\cite{highembdim}, as did a correct conjecture of the upper bound for embedding codimension 3.  Their conjecture was accompanied by a remark about how a proof with the same techniques would require ``a big amount of cases and subcases''.  
We~include a full proof here of the bounds proved in~\cite{highembdim} that avoids such casework, to contrast the use of Kunz nilsemigroup machinery with that of the original manuscript.  
\end{remark}

\begin{thm}\label{t:upperbound}
Suppose $N$ is a partly cancellative nilsemigroup with embedding codimension $r$.  If $r \le 2$, then $\eta(N) \le \binom{e}{2}$, and if $r = 3$, then $\eta(N) \le \binom{e}{2} + 1$.  As such, if $S$ is a numerical semigroup with embedding codimension $r$, then $\eta(S) \le \binom{e}{2}$ if $r \le 2$, and $\eta(S) \le \binom{e}{2} + 1$ if $r = 3$.  
\end{thm}

\begin{proof}
Let $n_1, \ldots, n_e$ denote the atoms of $N$.  
If $r = 1$, the factorizations of the only nonzero non-nil non-atom $p \in N$ all have coordinate sum 2 and pairwise disjoint support, meaning an outer Betti element $B$ can only be divisible by $p$ if $B = \{3e_i\}$ for some $i$.  As such, $\eta(N) \le \eta(N/p) = \binom{e}{2}$ by Lemma~\ref{l:lowerbound}.  

Next, suppose $r = 2$, and let $p, q \in N$ denote the nonzero non-nil non-atoms.  If~$q$ has a factorization with coordinate sum 3, then it is the only non-nil element of $N$ with this property.  As such, the only outer Betti element of $N$ that $q$ can divide is one of the form $\{4e_i\}$ for some $i$, so Lemma~\ref{l:lowerbound} implies $\eta(N) \le \eta(N/q) \le \binom{e}{2}$.  
If, on the other hand, neither $p$ nor $q$ has a factorization with coordinate sum 3, then any outer Betti element $B$ with a coordinate sum 3 factorization $z$
with $|\supp(z)| \ge 2$ 
must have $z - e_i \in \mathsf Z(p)$ or $z - e_i \in \mathsf Z(q)$ for each $i \in \supp(z)$.  By the connectivity of $\nabla_B$, we must have $\supp(B) = \supp(p) \cup \supp(q)$, which in particular means $|\supp(B)| = 2$.  This forces $|\mathsf Z_N(p)| = |\mathsf Z_N(q)| = |B| = 1$, so we have $\eta(N) = \binom{e}{2} - 1$.  
In all other cases, $p$~and $q$ each divide at most one outer Betti element, so $\eta(N) \le \eta(N/p) \le \binom{e}{2}$.  

Lastly, suppose $r = 3$, let $p, q, t \in N$ denote the nonzero non-nil non-atoms. 
As~a consequence of partial cancellativity, the support set of any element of $N$ must contain the support sets of its divisors, so if $p$ has a factorization with coordinate sum at least 3, then it is the only element of $N$ with this property.  This means $p$ can only divide outer Betti elements of the form $\{4e_i\}$ for some $i$, and Lemma~\ref{l:lowerbound} implies $\eta(N) \le \eta(N/p) \le \binom{e}{2}$.  As such, assume all factorizations of $p$, $q$, and $t$ have coordinate sum 2.  
If some outer Betti element $B$ is divisible by $p$, $q$, and $t$, then the connectivity of $\nabla_B$ implies $|\supp(B)| = 3$, but this forces 
$\mathsf Z_N(p) = \{e_i + e_j\}$, $\mathsf Z_N(q) = \{e_i + e_k\}$, $\mathsf Z_N(t) = \{e_j + e_k\}$, and $B = \{e_i + e_j + e_k\}$.  This means $B$ is the only outer Betti element divisible by $p$, so again by Lemma~\ref{l:lowerbound} we are done.  On the other hand, if an outer Betti element $B$ is divisible by $p$ and $q$ but not $t$, then by the argument in the second half of the preceding paragraph, $|\mathsf Z_N(p)| = |\mathsf Z_N(q)| = |B| = 1$.  This means at most~2 outer Betti elements are divisible by $p$, so $\eta(N) \le \eta(N/p) + 1 \le \binom{e}{2} + 1$.  
\end{proof}

\begin{remark}\label{r:generalnilsemigroups2}
Theorem~\ref{t:upperbound} illustrates another advantage of reformulating questions about $\eta$ in terms of finite partly cancellative nilsemigroups.  In this case, the quotient construction introduced in Definition~\ref{d:maximalelement} consolidates much of the casework seen in the argument in~\cite{highembdim} for $r \in [0,2]$.  One may also notice from the proof of Theorem~\ref{t:upperbound} that the value of $\eta(S)$ seems to be maximized when every element of $\Ap(S)$ has a unique factorization, which in turn forces all outer Betti elements of the Kunz nilsemigroup~$N$ to be singletons; we revisit this idea in Conjecture~\ref{conj:uniqueaperyupperbound}.  
\end{remark}

\begin{example}\label{e:extrabettifamily}
The core of the argument in the proof of Theorem~\ref{t:upperbound} for $r \le 2$ is that in almost all cases, outer Betti elements, together with the factorizations appearing in relations at non-nil elements of $N$, must have distinct support sets with cardinality at most 2.  When $r = 3$, this need no longer be the case.  For example, $S = \<7, 15, 17, 33\>$, whose Kunz poset is depicted in Figure~\ref{fig:extrabettifamily}, has outer Betti elements 
$$
B_2 = \{(0,2,1,0)\}
\qquad \text{and} \qquad
B_3 = \{(0,1,2,0)\}
$$
with identical support.  
\end{example}

\begin{figure}[t]
\begin{center}
\includegraphics[height=1.5in]{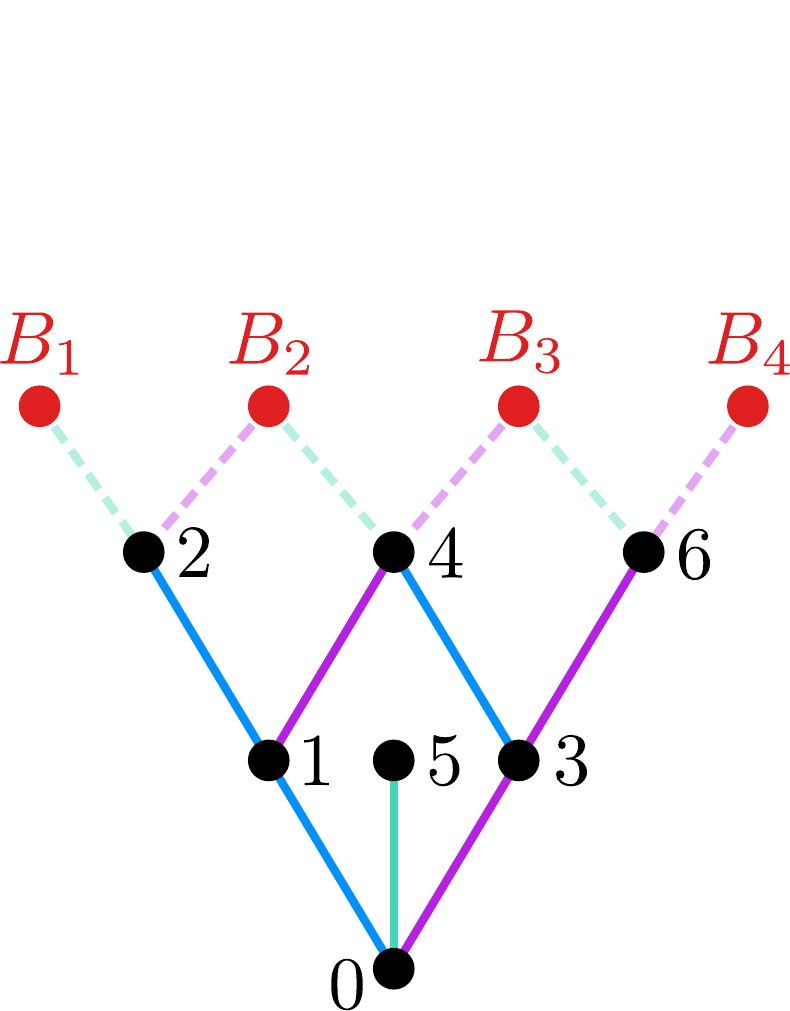}
\hspace{3em}
\includegraphics[height=1.5in]{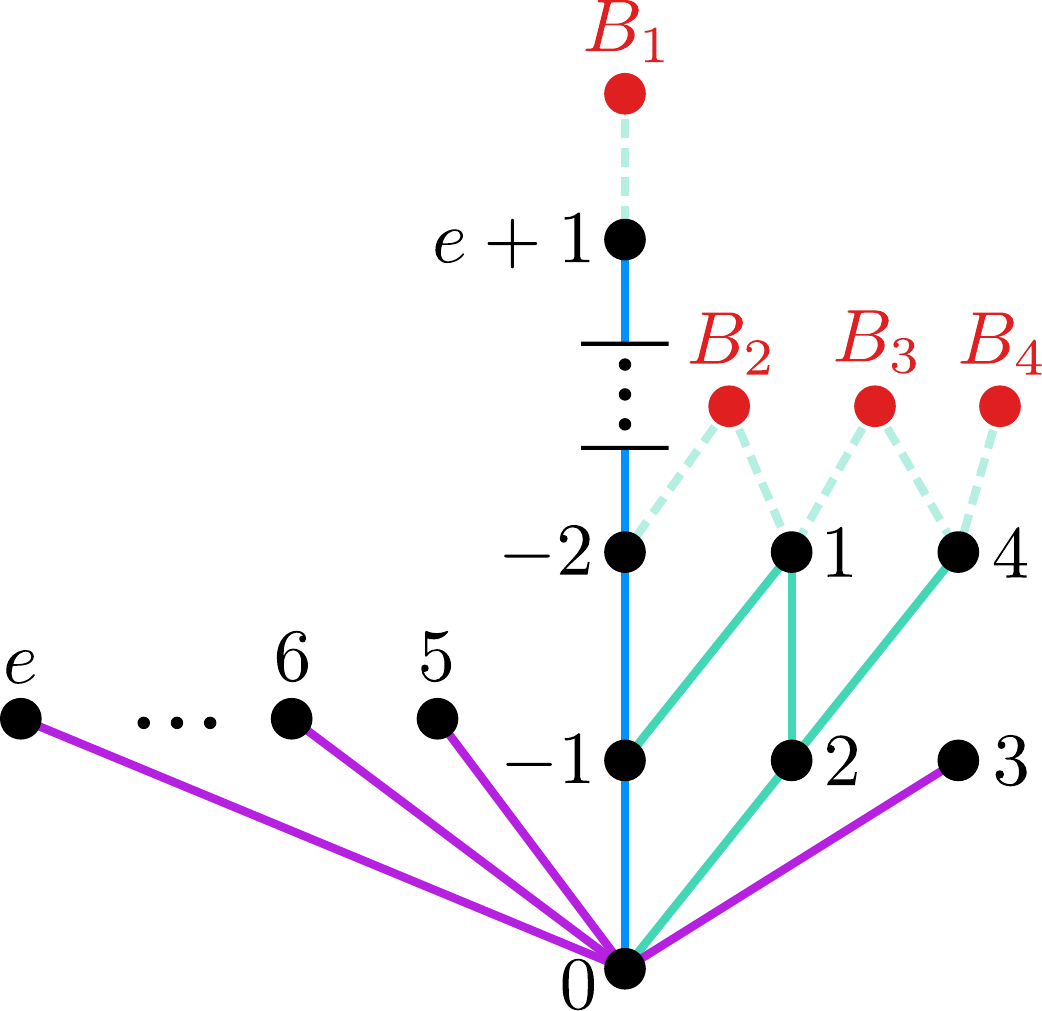}
\end{center}
\caption{The Kunz poset structure of the numerical semigroup from Example~\ref{e:extrabettifamily} (left) and that of $S$ with $\eta(S) = \binom{e}{2} + 1$ from Theorem~\ref{t:extrabettifamily} (right).}
\label{fig:extrabettifamily}
\end{figure}

The following theorem assumes $e \ge 5$, as a slightly different argument is needed for $e = 4$.  The latter case will be handled in Section~\ref{sec:embdim4}, as part of a much larger family presented in Theorem~\ref{t:embdim4}.  

\begin{thm}\label{t:extrabettifamily}
If $e \ge 5$ and $r \ge 3$, then there exists a numerical semigroup $S$ with embedding dimension $e$ and multiplicity $m = e + r$ such that $\eta(S) = \binom{e}{2} + 1$.  
\end{thm}

\begin{proof}
Consider the numerical semigroup
$$
S = \<m, 2m - 1, (r - 1)m + 2, (2r - 2)m + i : i \in \{3\} \cup [5, e] \>,
$$
and let $a_i = n_i$ denote the generator of $S$ with $n_i \equiv i \bmod m$.  
For each $j \in [0, r - 1]$, we have $a_j = jn_{-1} = 2jm - j \in \Ap(S)$, while
$$
rn_{-1} = (2r - 1)m + e > a_e.
$$
The remaining elements of $\Ap(S) = \{a_i : i \in \ZZ_m\}$ are 
$$ 
a_1 = n_2 + n_{-1} = (r+1)m + 1
\qquad \text{and} \qquad
a_4 = 2n_2 = (2r - 2)m + 4.
$$
as each equivalence class modulo $m$ is accounted for, any non-negative integer combinations of 2 or more generators involving $n_3$ or $n_i$ with $i \in [5,e]$ exceeds $a_e = \max \Ap(S)$, $2n_2 + n_{-1} > a_3$, and $n_2 + 2n_{-1} \equiv 0 \bmod m$.  This proves $S$ has the Kunz poset depicted in Figure~\ref{fig:extrabettifamily}, whose outer Betti elements are $B_1, \ldots, B_4$ along with any coordinate sum 2 factorization involving $n_3$ or $n_i$ with $i \in [5,e]$.  By Theorem~\ref{t:kunzfaces3}, $\eta(S) = \binom{e}{2} + 1$.  
\end{proof}

\section{Minimal presentation cardinalities in embedding dimension 4}
\label{sec:embdim4}

We now turn our attention to more thoroughly characterizing achievable minimal presentation cardinalities for $e = 4$. 

\begin{remark}\label{r:eta3}
In Section~\ref{sec:intervalfamily}, we demonstrated a family of numerical semigroups which achieves any minimal presentation cardinality in the range $[\binom{e-1}{2}+1, \binom{e}{2}] = [4, 6]$.  Using methods from~\cite{wilfmultiplicity} and~\cite{kunzfaces3}, one~can verify computationally that there are no numerical semigroups with $m = 7$ and $\eta = 3$, and Proposition~\ref{p:eta3} ensures there are achievable such numerical semigroups for any $m \ge 8$.  This~completely characterizes which values of $\eta \le 6$ are attained by a numerical semigroup for $e = 4$ and every $m$.  
\end{remark}

\begin{prop}\label{p:eta3}
For any $m \geq 8$, there is a numerical semigroup $S$ with $m(S) = m$, $e(S) = 4$, and $\eta(S) = 3$.
\end{prop}

\begin{proof}
Note that the numerical semigroup $T = \<4,5,6\>$ has $m \in T$ as a non-atom.  Let $a$ be any prime greater than $m$. Then $4a > m$ and $m$ does not divide any of $4a$, $5a$, or $6a$.  The gluing $S = m\ZZ_{\ge 0} + aT$ then has $m(S) = m$, $e(S) = 1 + e(T) = 4$ and $\eta(S) = \eta(T) + 1 = 3$, as desired. 
\end{proof}

Having fully characterized $\eta \le 6$, we now turn our attention to $\eta \ge 6$.

\begin{figure}[t]
\begin{center}
\includegraphics[height=2.25in]{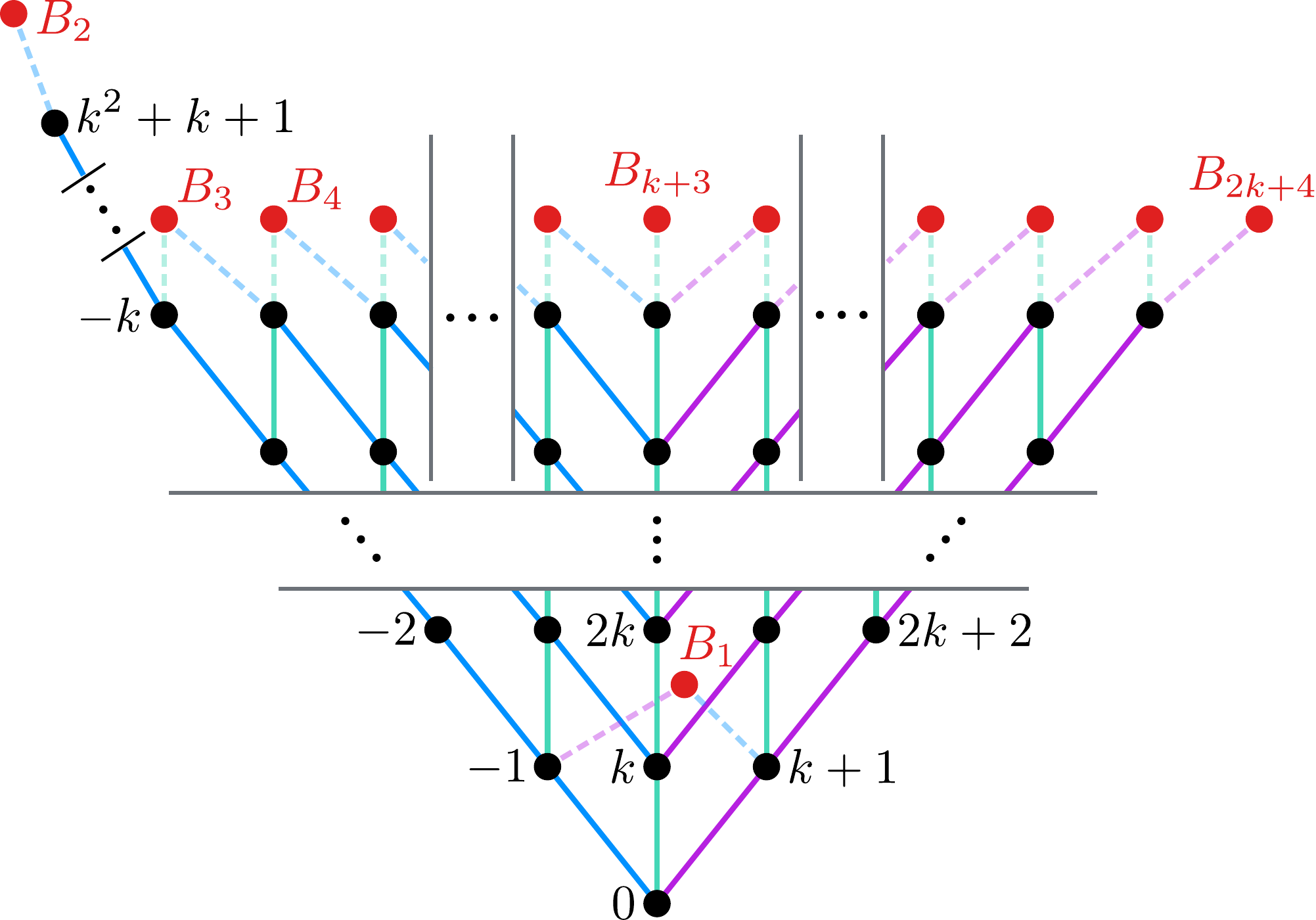}
\hspace{-1em}
\includegraphics[height=2.25in]{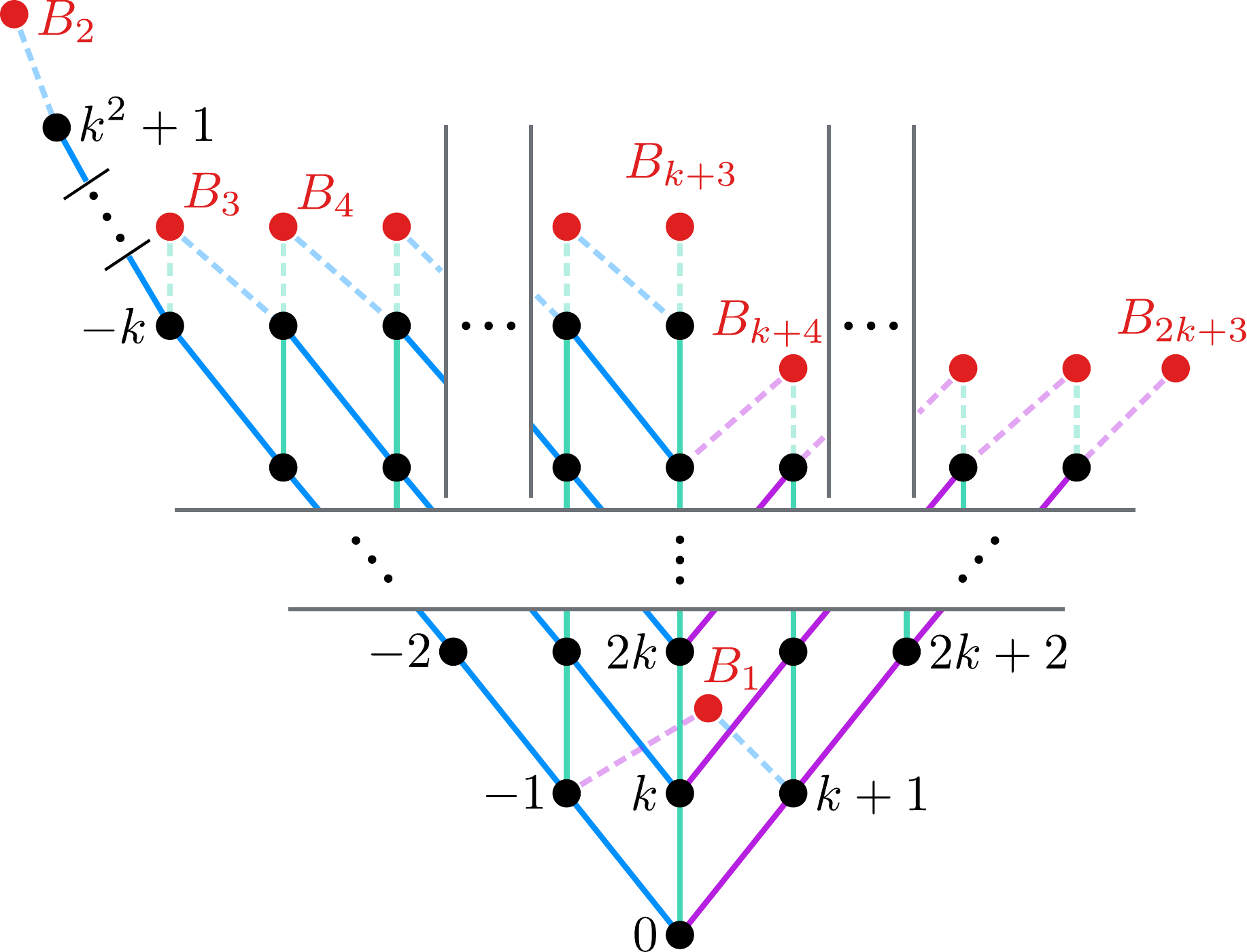}
\end{center}
\caption{The Kunz poset structure for $S$ with $\eta(S) = 2k + 4$ (left) and $\eta(S) = 2k + 3$ (right) in the proof of Theorem~\ref{t:embdim4}.}
\label{fig:embdim4}
\end{figure}

\begin{thm}\label{t:embdim4}
For any $\eta \geq 6$ and $m \in \mathbb{N}$ with $4m \geq (\eta-2)^2$, there exists a numerical semigroup $S$ with $e(S) = 4$, $\eta(S) = \eta$, and $m(S) = m$.
\end{thm}

\begin{proof}
We proceed by cases, based on the parity of $\eta$.  First, suppose $\eta = 2k + 4$ for some $k \in \ZZ_{\ge 1}$, so that $m \geq (k+1)^2$.  We claim the Kunz nilsemigroup of
$$S = \left\langle m,(k+1)m - 1,  (m-k^2-k)m +  k, (m-k^2-k)m +  k+1\right\rangle$$
is the one whose poset $P$ is depicted in Figure~\ref{fig:embdim4}.  More specifically, we claim each element of $\Ap(S)$ has a unique factorization, each lying in the set 
$$A = \{(0, j, i-j, 0), (0, 0, i-j, j) : 0 \le j \le i \le k\} 
\cup \{(0, j, 0, 0) : k < j < m - k^2 - k\}.$$
As there are $2i + 1$ elements of $A$ with coordinate sum $i \le k$, it is easy to check that
$$|A| = \sum_{i=0}^{k} (2i+1) + (m - k^2 - 2k - 1) = m.$$
Additionally, each $\varphi(z)$ lies a distinct equivalence class modulo $m$ for each $z \in A$.  Indeed, if $(0,a,0,0), (0,0,b,c) \in A$ with $\varphi(0,a,0,0) \equiv \varphi(0,0,b,c) \bmod m$, then 
$$\varphi(0,0,b,c) \equiv k(b + c) + c \bmod m
\qquad \text{and} \qquad
\varphi(0,a,0,0) \equiv -a \bmod m,$$
the right hand sides of which lie in $[0, k^2 + k]$ and $[-m + k^2 + k + 1, 0]$, respectively, ensuring $a = b = c = 0$.  Meanwhile, if $(0,a,b,0), (0,0,0,c) \in A$ with $b > 0$ satisfy $\varphi(0,a,b,0) \equiv \varphi(0,0,0,c) \bmod m$, then 
$$kb - a \equiv \varphi(0,a,b,0) \equiv \varphi(0,0,0,c) \equiv kc + c \bmod m$$
necessitates $kb - a = kc + c$, and thus either (i) $b = c + 1$ and $a + c = k$, which is impossible since then $a + b = a + c + 1 = k + 1$, or (ii) $b = c + 2$ and $a = c = k$, which is impossible since then $b = k + 2$.  

It remains to show any $z \in \ZZ_{\ge 0}^4 \setminus A$ with first coordinate 0 satisfies $\varphi(z) \notin \Ap(S)$.  It suffices to assume $z$ is minimal in $\ZZ_{\ge 0}^4 \setminus A$ under the componentwise partial order, and thus corresponds to a proposed outer Betti element $B_i$ of $P$ depicted in Figure~\ref{fig:embdim4}.  
Comparing equivalence classes modulo $m$, one can then readily check:\ if $i = 1$, then $z = (0,1,0,1)$ and
$$\varphi(0,1,0,1) = (m - k^2 + 1)m + k > (m - k^2 - k)m + k = \varphi(0,0,1,0);$$
if $i = 2$, then $z = (0, m - k^2-k, 0, 0)$ and 
\begin{align*}
    \varphi(0, m - k^2 - k, 0, 0) 
    &= ((m - k^2 - k)(k+1) - 1)m + k^2 + k \\
    &> k(m-k^2-k)m +  k^2 + k = \varphi(0, 0, 0, k);
\end{align*}
if $3 \le i \le k + 3$, then $z = (0, k+1-b, b, 0)$ with $b = i - 2 \in [1, k+1]$ and 
$$\varphi(0, k + 1 - b, b, 0) \ge \varphi(0, 0, b, 0) > \varphi(0, 0, 0, b - 1);$$
and if $k + 4 \le i \le 2k + 4$, then $z = (0, 0, k + 1 - c, c)$ with $c = i - k - 3 \in [1, k+1]$ and
\begin{align*}
    \varphi(0,0,k+1-c,c) 
    &= (k+1)(m-k^2-k)m + k^2 + k + c \\
    &> ((k+1)(m -k^2-k-c) - 1)m + k^2 + k + c \\
    &= (m -k^2-k-c)((k+1)m-1) = \varphi(0,m - k^2-k - c,0,0).
\end{align*}
One may now simply count the outer Betti elements of $P$ to obtain $\eta(S) = 2k + 4$.


Next, suppose $\eta = 2k+3$ for $k \in \ZZ_{\ge 2}$, so that $m \ge k^2+k+1$.  We claim the Kunz nilsemigroup of $$S = \left\langle m, km - 1,   (m - k^2-1)m +  k, (m - k^2 -1)m +  k+1\right\rangle$$ is the one whose poset $P$ is depicted in Figure~\ref{fig:embdim4}. More specifically, we claim each element of $\Ap(S)$ has a unique factorization, each lying in the set 
\begin{align*}
A &= \{(0, j, i-j, 0) : 0 \le j \le i \le k\}
\cup \{(0, 0, i-j, j): 1 \le j \le i \le k-1\} \\
& \qquad \cup \{(0, j, 0, 0) : k < j \le m - k^2 -1\}.
\end{align*}
As there are $2i+1$ elements of $A$ with coordinate sum $i \leq k-1$, we see
$$|A| = \sum_{i=0}^{k-1} (2i+1) + (k + 1) + (m-k^2-k - 1) = m.$$ Additionally, we see that $\varphi(z)$ lies is a distinct equivalence class modulo $m$ for each $z \in A$. Indeed, if $(0,a,0,0), (0,0,b,c) \in A$ with $\varphi(0,a,0,0) \equiv \varphi(0,0,b,c) \bmod m$, then 
$$\varphi(0,0,b,c) \equiv k(b + c) + c \bmod m
\qquad \text{and} \qquad
\varphi(0,a,0,0) \equiv -a \bmod m,$$
the right hand sides of which lie in $[0, k^2]$ and $[-m + k^2 + 1, 0]$, respectively, thereby ensuring $a = b = c = 0$. Meanwhile, if $(0,a,b,0), (0,0,0,c) \in A$ with $b > 0$ satisfy $\varphi(0,a,b,0) \equiv \varphi(0,0,0,c) \bmod m$, then 
$$kb - a \equiv \varphi(0,a,b,0) \equiv \varphi(0,0,0,c) \equiv kc + c \bmod m$$
necessitates $kb - a = kc + c$, and thus either (i) $b = c + 1$ and $a + c = k$, which is impossible since then $a + b = a + c + 1 = k + 1$, or (ii) $b = c + 2$ and $a = c = k$, which is impossible since then $b = k + 2$.  

Proceeding as before, it remains to show that for each proposed outer Betti element $B_i = \{z\}$ of $P$, we have $\varphi(z) \notin \Ap(S)$.  
If $i = 1$, then $z = (0,1,0,1)$ and
$$\varphi(0,1,0,1) = (m - k^2 + k + 1)m + k > (m - k^2-1)m +  k = \varphi(0,0,1,0);$$
if $i = 2$, then $z = (0,m - k^2,0,0)$ and 
\begin{align*}
    \varphi(0, m - k^2, 0, 0) 
    &= (km - k^3 - 1)m + k^2 
    > k(m-k^2-1)m +  k^2 = \varphi(0, 0, 0, k);
\end{align*}
if $3 \le i \le k + 2$, then $z = (0, k+1-b, b, 0)$ with $b = i - 2 \in [1, k+1]$ and 
$$\varphi(0, k + 1 - b, b, 0) \ge \varphi(0, 0, b, 0) > \varphi(0, 0, 0, b - 1);$$
if $i = k+3$, then $z = (0,0,k+1,0)$ and 
\begin{align*}
    \varphi(0,0, k+1,0) &= (k+1)(m - k^2 - 1)m + k^2 + k\\
    &> (km - k^3 - k^2 - 1)m + k^2 + k = \varphi(0,m-k^2-k,0,0);
\end{align*}
and if $k+4 \le i \le 2k+3$, then $z = (0,0,k-c,c)$ with $c = i - k - 3 \in [1,k]$ and 
\begin{align*}
    \varphi(0,0,k-c,c) &= (m-k^2+k)km + k^2 + c\\
    &> (km - k^3 - ck - 1)m + k^2 + c\\
    &= (m - k^2 - c)(km - 1) = \varphi(m - k^2 - c,0,0).
\end{align*}
Counting the outer Betti elements of $P$ yields $\eta(S) = 2k + 3$, as desired.
\end{proof}

\section{Some open questions}\label{sec:openquestions}
\label{sec:extras}

As noted in the introduction, we have verified computationally that every value of~$\eta$ attained by a numerical semigroup $S$ with $e = 4$ and $m \le 42$ is accounted for by the families presented in this manuscript.  As such, we conjecture the following.  

\begin{conj}\label{conj:embdim4}
For $e = 4$, the families of numerical semigroups in Theorems~\ref{t:intervalfamily} and~\ref{t:embdim4} and Proposition~\ref{p:eta3} attain all possible values of $\eta$ for each $m \ge 4$.  
\end{conj}

We suspect that such a complete answer for $e \ge 5$ will be much more difficult, as maximizing the number of outer Betti elements for Kunz nilsemigroups with a given number of atoms is similar to problems in the field of additive bases, wherein tight bounds are notoriously difficult to obtain in general~\cite{newupperboundfiniteadditivebases,extremalbasesfinitecyclic,upperboundfiniteadditive2bases}.  

Turning our attention to arbitrary $e$, in light of the lower bound in Theorem~\ref{t:lowerbound}, the following is a natural question.  

\begin{question}\label{q:etalarge}
For fixed $e$ and $m$, what is the largest $\eta$ can be?  
\end{question}

The following would be a good first step towards Question~\ref{q:etalarge}.  

\begin{conj}\label{conj:uniqueaperyupperbound}
For fixed $e$ and $m$, the largest possible value of $\eta$ is achieved by a numerical semigroup $S$ in which every element of $\Ap(S)$ has a unique factorization.  In this case, each outer Betti element of the Kunz nilsemigroup $N$ is a singleton, and $\eta(S)$ equals the number of outer Betti elements of $N$.  
\end{conj}

As identified at the end of the introduction, a consequence of the results in this paper is that an answer to the following question would yield, for each~$m$, a completely characterization of the attainable values of $\eta$ across all embedding dimensions.  

\begin{question}\label{q:smallcodimupper}
If $r \in [4, e]$, which values of $\eta$ are attained with $\binom{e}{2} + 2 \le \eta \le \binom{e}{2} + e - r$?  
\end{question}

On the other hand, the lower bound in Theorem~\ref{t:lowerbound} is known to be sharp for $m < 2e$ by the family in Theorem~\ref{t:intervalfamily}, but Figure~\ref{fig:achieved} indicates it can also be sharp for larger $m$.  

\begin{question}\label{q:lowerbound}
For which $m$ and $e$ is the lower bound in Theorem~\ref{t:lowerbound} sharp?  
\end{question}

If $\eta(S) = \mathsf e(S) - 1$, then $S$ is \emph{complete intersection} and can be constructed via successive gluings (see~\cite{completeintersection}).  The following lower bound for the left-most column of each outlined region in Figure~\ref{fig:achieved} is a first step in the direction of Question~\ref{q:lowerbound}.  

\begin{prop}\label{p:completeintersection}
Any complete intersection numerical semigroup $S$ with $\mathsf m(S) = m$ and $\mathsf e(S) = e$ satisfies $m \ge 2^{e-1}$.  
\end{prop}

\begin{proof}
We proceed by induction on $e$.  If $e = 1$, then $S = \ZZ_{\ge 0}$, so suppose $e \ge 2$.  
If $S$ is complete intersection with $\mathsf m(S) = m$ and $\mathsf e(S) = e$, then $S = aT + a'T'$ for some complete intersection $T$ and $T'$ with $\mathsf e(T) + \mathsf e(T') = e$.  This means
$$m = \min(a \mathsf m(T), a' \mathsf m(T')) \ge \min(2^{\mathsf e(T')} 2^{\mathsf e(T)-1}, 2^{\mathsf e(T)} 2^{\mathsf e(T)-1}) = 2^{e-1}$$
since $a' \ge 2\mathsf m(T)$ and $a \ge 2\mathsf m(T')$ are non-atoms in $T$ and $T'$, respectively.  
%
%
\end{proof}

If one fixes $m$ and $e$, then the set of attainable values of $\eta$ is often an interval, but need not be in general, as the red outline on the right hand side of Figure~\ref{fig:achieved} indicates for $m = 13$ and $e = 7$.  In fact, $\eta = 25$ is only attained by such a numerical semigroup if its Kunz poset is, up to symmetry, one of the two depicted in Figure~\ref{fig:noninterval}.  Notice in particular there are 3 atoms for which the sum of any two lies in the Ap\'ery set.  On the other hand, $\eta = 24$ is not attained by any such numerical semigroup.  

\begin{question}\label{q:wheninterval}
When is the set of attainable values of $\eta$ an interval for fixed $e$ and $m$?  
\end{question}

\begin{figure}[t]
\begin{center}
\includegraphics[height=1.5in]{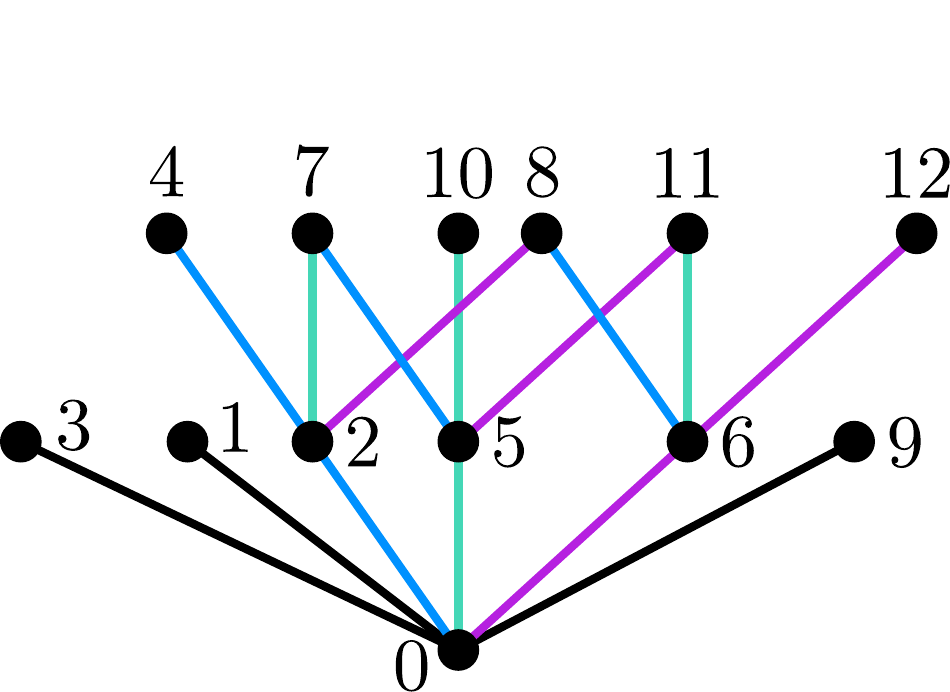}
\hspace{3em}
\includegraphics[height=1.5in]{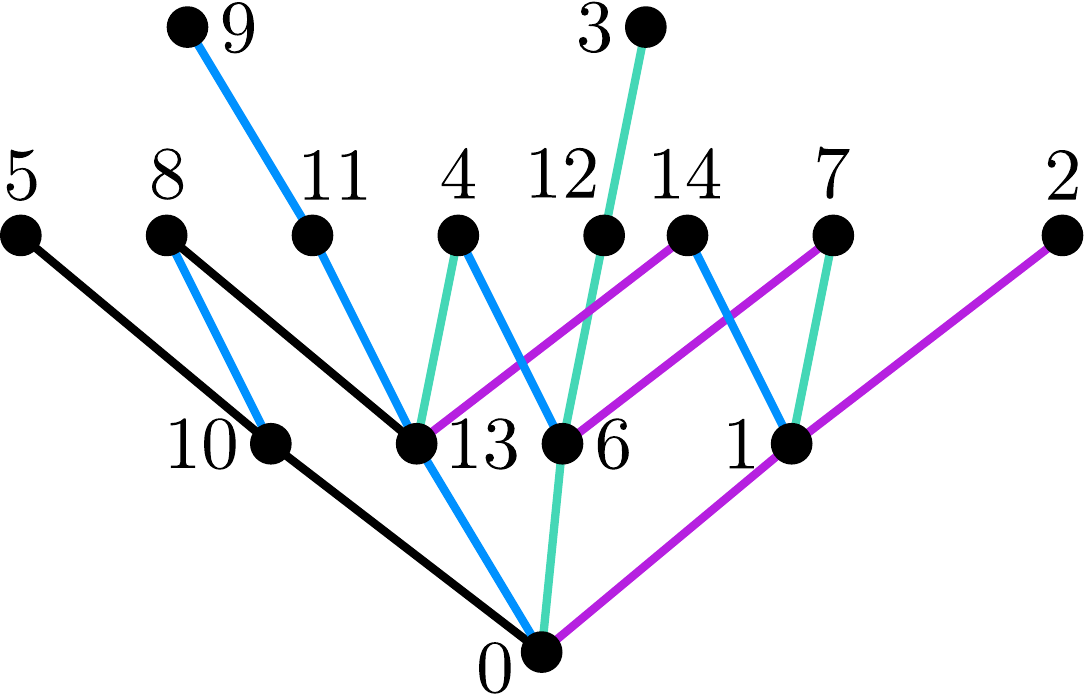}
\end{center}
\caption{Kunz posets of numerical semigroups with $m = 13$, $e = 7$, and $\eta = 25$ (left) and $m = 15$, $e = 5$, $\eta = 15$ (right).}
\label{fig:noninterval}
\end{figure}

Likewise, if one fixes $e$, then the set of values of $m$ for which a fixed $\eta$ can be attained need not be an interval.  Indeed, Figure~\ref{fig:achieved} indicates there exists a numerical semigroup with $e = 5$ and $\eta = 15$ for $m = 15$ (a sample Kunz poset is given in Figure~\ref{fig:noninterval}), but not for $m = 16$. Note that if Conjecture~\ref{conj:embdim4} is true, then for $e = 4$ the achievable values when fixing either $m$ or $\eta$ form an interval.


In this direction, Proposition~\ref{p:verticalinterval} implies that, for any $e \geq 4$ and $\eta \geq e-1$, there exists a numerical semigroup with embedding dimension $e$ and minimal presentation cardinality $\eta$ for all but finitely many multiplicities. In particular, there are at most finitely many such ``jumps'' in each column of Figure~\ref{fig:achieved}.

\begin{prop}\label{p:verticalinterval}
If $e \ge 4$ and $\eta \ge e - 1$, there is an $M \in \ZZ$ such that, for all $m \geq M$, there exists a numerical semigroup $S$ with $\mathsf e(S) = e$, $\eta(S) = \eta$, and $\mathsf m(S) = m$. 
\end{prop}

\begin{proof}
The claim holds for $e = 4$ by Theorem~\ref{t:embdim4}.  As such, suppose $e \ge 5$ and $\eta \ge e - 1$.  By induction on $e$, we may suppose $S'$ is a numerical semigroup with $\mathsf e(S') = e - 1$ and $\eta(S') = \eta - 1$.  Let $f = \max(\ZZ_{\geq0} \setminus S)$ and let $M = \mathsf m(S') + f$.  Then for any $m \ge M$, the numerical semigroup $S = m\ZZ_{\ge 0} + (m+1)S'$ is a gluing of $S'$ with $\ZZ_{\ge 0}$ that has $\mathsf e(S) = e$, $\eta(S) = \eta(S') + 1 = \eta$, and $\mathsf m(S) = m$.  
\end{proof}


The bound $M$ obtained in the proof of Proposition~\ref{p:verticalinterval} is far from optimal.  For example, if $e = 5$ and $\eta = 10$ the minimal such bound is $M = 5$, but when applying the construction in the proof, one obtains $M = 126$. This raises the following question.  

\begin{question}\label{q:verticalinterval}
In terms of $e$ and $\eta$, what is the minimal value of $M$ in Proposition~\ref{p:verticalinterval}?
\end{question}

\section*{Acknowledgements}
Much of this work was completed during the 2021 PolyMath REU.  
The authors made use of the computer software packages~\cite{numericalsgpsgap,kunzposetsage,numericalsgpssage} throughout their work.  




\begin{thebibliography}{HHHKR10}
\raggedbottom

\bibitem{numericalappl}
A.~Assi and P.~Garc\'ia-S\'anchez,
\emph{Numerical semigroups and applications},
RSME Springer Series, 1.~Springer, [Cham], 2016.

\bibitem{kunzfaces2}
J.~Autry, A.~Ezell, T.~Gomes, C.~O'Neill, C.~Preuss, T.~Saluja, and E.~Torres-Davila, 
\emph{Numerical semigroups, polyhedra, and posets II:\ locating certain families of semigroups},
Advances in Geometry \textbf{22} (2022), no.~1, 33--48.  
Available at \textsf{arXiv:1912.04460}.

\bibitem{primeidealsgenericzero}
H.~Bresinsky, 
\emph{On prime ideals with generic zero $x_i = t^{n_i}$},
Proc.~Amer.\ Math.~Soc.~47 (1975), no.~2, 329--332.

\bibitem{wilfmultiplicity}
W.~Bruns, P.~Garc\'ia-S\'anchez, C.~O'Neill, and D.~Wilburne,
\emph{Wilf's conjecture in fixed multiplicity}, 
International Journal of Algebra and Computation \textbf{30} (2020), no.~4, 861--882.

\bibitem{monomialcurvecmtype}
M.~Cavaliere and G.~Niesi,
\emph{On monomial curves and Cohen-Macaulay type},
Manuscripta Math. 42 (1983), 147--159. 

\bibitem{reesquotientsns}
M.~Delgado and V.~Fernandes,
\emph{Rees quotients of numerical semigroups},
Port.~Math.~70 (2013), no.~2, 93--112.

\bibitem{numericalsgpsgap}
M.~Delgado, P.~Garc\'ia-S\'anchez, and J.~Morais, 
\emph{NumericalSgps, A package for numerical semigroups}, 
Version 1.1.10 (2018), (Refereed GAP package),
\url{https://gap-packages.github.io/numericalsgps/}.

\bibitem{delormegluings}
C.~Delorme, 
\emph{Sous-mono\"ides d'intersection compl\'ete de N},
Ann. Sci. \'Ecole Norm. Sup. (4) \textbf{9} (1976), no.~1, 145--154.

\bibitem{completeintersection}
P.~Garc\'ia-S\'anchez and J.~Rosales, 
\emph{On complete intersection affine semigroups}, 
Comm.\ Algebra \textbf{23} (1995), no.~14, 5395--5412.

\bibitem{kunzposetsage}
T.~Gomes, C.~O'Neill, C.~Preuss, and E.~Torres Davila, 
\emph{KunzPoset}
(Sage software),
\url{https://github.com/coneill-math/kunzpolyhedron}.

\bibitem{kunzfaces3}
T.~Gomes, C.~O'Neill, and E.~Torres-Davila, 
\emph{Numerical semigroups, polyhedra, and posets III:\ minimal presentations and face dimension},
preprint.  
Available at \textsf{arXiv:2009.05921}.

\bibitem{newupperboundfiniteadditivebases}
C.~G\"unt\"urk and M.~Nathanson, 
\emph{A new upper bound for finite additive bases},
Acta Arith.\ 124 (2006), no.~3, 235--255.

\bibitem{extremalbasesfinitecyclic}
X.~Jia, and J.~Shen,
\emph{Extremal bases for finite cyclic groups},
SIAM J.~Discrete Math.\ 31 (2017), no.~2, 796--804.

\bibitem{kunzfaces1}
N.~Kaplan and C.~O'Neill, 
\emph{Numerical semigroups, polyhedra, and posets I:\ the group cone},
Combinatorial Theory \textbf{1} (2021), \#19.
Available at \textsf{arXiv:1912.03741}.

\bibitem{compapery}
G.~M\'arquez-Campos, I.~Ojeda, and J.~Tornero, 
\emph{On the computation of the Apéry set of numerical monoids and affine semigroups}
Semigroup Forum 91 (2015), no.~1, 139--158.

\bibitem{numericalsgpssage}
C.~O'Neill, 
\emph{numsgps-sage}
(Sage software), 
\url{https://github.com/coneill-math/numsgps-sage}.

\bibitem{genericlatticecodim3}
I.~Ojeda,
\emph{Examples of generic lattice ideals of codimension 3},
Comm. Algebra 36 (2008), no.~1, 279--287.

\bibitem{shortresolutionalg}
I.~Ojeda and A.~Vigneron-Tenorio, 
\emph{The short resolution of a semigroup algebra},
Bull.\ Aust.\ Math.\ Soc.\ 97 (2017), no.~3, 400--411. 

\bibitem{rosalesApery}
J.~Rosales, 
\emph{Numerical semigroups with Ap\'ery sets of unique expression},
Journal of Algebra 226, 479–487, 2000.

\bibitem{fingenmon}
J.~Rosales and P.~Garc\'ia-S\'anchez, 
\emph{Finitely generated commutative monoids}, 
Nova Science Publishers, Inc., Commack, NY, 1999. xiv+185 pp.~ISBN: 1-56072-670-9.  

\bibitem{numerical}
J.~Rosales and P.~Garc\'ia-S\'anchez, 
\emph{Numerical {S}emigroups}, 
Developments in Mathematics, Vol. 20, Springer-Verlag, New York, 2009.

\bibitem{highembdim}
J.~Rosales and P.~García-Sánchez, 
\emph{On numerical semigroups with high embedding dimension},
J.~Algebra 203 (1998), no.~2, 567--578.

\bibitem{grobpoly}
B.~Sturmfels,
\emph{Gr\"obner bases and convex polytopes},
University Lecture Series, 8. American Mathematical Society, 
Providence, RI, 1996. xii+162 pp. ISBN: 0-8218-0487-1.

\bibitem{nsbettisurvey}
D.~Stamate, 
\emph{Betti numbers for numerical semigroup rings},
Multigraded algebra and applications, 133--157,
Springer Proc.\ Math.\ Stat., 238, Springer, Cham, 2018.

\bibitem{upperboundfiniteadditive2bases}
G.~Yu, 
\emph{Upper bounds for finite additive 2-bases},
Proc.~Amer.\ Math.~Soc.~137 (2009), no.~1, 11--18.

\end{thebibliography}
\end{document}